\documentclass{elsarticle}
\usepackage{graphicx}
\usepackage{amscd,amsmath,amstext,amsfonts,amsbsy,amssymb,amsthm,eufrak}
\usepackage{hyperref}
\usepackage[toc,page]{appendix}
\usepackage{chngcntr}

\newtheorem{theo}{Theorem}
\newtheorem{lemm}[theo]{Lemma}

\newtheorem{prop}[theo]{Proposition}

\newdefinition{defi}{Definition}

\newdefinition{rema}{Remark}

\makeatletter
\def\ps@pprintTitle{
 \let\@oddhead\@empty
 \let\@evenhead\@empty
 \def\@oddfoot{}%
 \let\@evenfoot\let\@oddfoot }
\makeatother

\newcommand{\mL}{\mathcal{L}}
\newcommand{\K}{\mathcal{K}}
\newcommand{\R}{\mathbb{R}}
\newcommand{\Rn}{\mathbb{R}^n}
\newcommand{\lag}{\left\langle}
\newcommand{\rag}{\right\rangle}

\begin{document}
\title{\large{\bf Construction of a stable blow-up solution for a class of strongly perturbed semilinear heat equations}}
\author{V. T. Nguyen and H. Zaag \footnote{This author is supported by the ERC Advanced Grant no. 291214, BLOWDISOL and by the ANR project ANA\'E ref. ANR-13-BS01-0010-03.} \\ \textit{\small{Universit{\'e} Paris 13, Sorbonne Paris Cit{\'e},}}\\
\textit{\small{LAGA, CNRS (UMR 7539), F-93430, Villetaneuse, France.}} }
\begin{abstract}
We construct a solution for a class of strongly perturbed semilinear heat equations which blows up in finite time with a prescribed blow-up profile. The construction relies on the reduction of the problem to a finite dimensional one and the use of index theory to conclude.\\
 
\end{abstract}
\maketitle
\section{Introduction}
We are interested in the following nonlinear parabolic equation: 
\begin{equation}\label{equ:problem}
\left\{
\begin{array}{rcl}
u_t &=& \Delta u + |u|^{p-1}u + h(u), \\
u(0) &=& u_0 \in L^\infty(\mathbb{R}^n),
\end{array}
\right.
\end{equation}
where $u$ is defined for $(x,t) \in \mathbb{R}^n \times [0,T)$, $1 < p$ and $p < \frac{n+2}{n-2}$ if $n \geq 3$, the function $h$ is in $\mathcal{C}^1(\mathbb{R}, \mathbb{R})$ satisfying
\begin{equation}\label{equ:h}
j = 0,1,\;\;|h^{(j)}(z)| \leq M\left( \dfrac{|z|^{p - j}}{\log^{a}(2 + z^2)} + 1\right) \;\;\;\text{with}\;\; a >1,\; M > 0,
\end{equation}
or 
\begin{equation}\label{equ:h1}
h(z) = \mu\dfrac{|z|^{p-1}z}{\log^{a}(2 + z^2)}\;\;\; \text{with} \;\; a > 0, \; \mu \in \mathbb{R}.
\end{equation}

\noindent By standard results, the Cauchy problem for equation \eqref{equ:problem} can be solved in $L^\infty(\mathbb{R}^n)$. The solution $u(t)$ of \eqref{equ:problem} would exist either on $[0,+\infty)$ (global existence) or only on $[0,T)$, with $0 < T < +\infty$. In this case, we say that $u(t)$ blows up in finite time $T$, namely
$$\lim_{t \to T} \|u(t)\|_{L^\infty(\mathbb{R}^n)} = +\infty.$$
Here $T$ is called the blow-up time, and a point $x_0 \in \mathbb{R}^n$ is called a blow-up point if and only if there exist $(x_n, t_n) \to (x_0, T)$ such that $|u(x_n,t_n)| \to +\infty$ as $n \to +\infty$. In this paper, we are interested in the finite time blow-up for equation \eqref{equ:problem}.\\

When $h \equiv 0$, the blow-up result for equation \eqref{equ:problem} is largely well-understood. The existence of blow-up solutions has been proved by several authors (see Fujita \cite{FUJsut66}, Ball \cite{BALjmo77}, Levine \cite{LEVarma73}). We have a lots of results concerning the asymptotic blow-up behavior, locally near a given blow-up point (see Giga and Kohn \cite{GKcpam89}, Weissler \cite{WEIjde84}, Filippas, Kohn and Liu \cite{FKcpam92}, \cite{FLaihn93}, Herrero and Vel\'azquez \cite{HVcpde92}, \cite{HVasnsp92}, \cite{VELcpde92}, \cite{VELtams93}, Merle and Zaag \cite{MZcpam98}, \cite{MZgfa98}, \cite{MZma00}). The notion of asymptotic profile appears also in various papers (see Bricmont and Kupiainen \cite{BKnon94}, Merle and Zaag \cite{MZdm97}, Berger and Kohn \cite{BKcpam88}, Nguyen \cite{NG14num} for numerical studies).\\

Given $b$ a blow-up point of $u$, we study the behavior of $u$ near the singularity $(b,T)$ through the following \emph{similarity variables} introduced by Giga and Kohn \cite{GKcpam85, GKiumj87, GKcpam89}:
\begin{equation}\label{equ:simivariables}
y = \frac{x - b}{\sqrt{T-t}}, \quad s = -\log(T-t), \quad w_b(y,s) = (T - t)^\frac{1}{p-1}u(x,t),
\end{equation}
and $w_b$ satisfies for all $(y,s) \in \mathbb{R}^n \times [-\log T, +\infty)$,
\begin{equation}\label{equ:divw1}
\partial_sw_b = (\Delta  - \frac{y}{2}\cdot \nabla  + 1)w_b - \frac{p}{p-1} w_b + |w_b|^{p-1}w_b + e^{-\frac{ps}{p-1}}h \left(e^\frac{s}{p-1}w_b\right).
\end{equation}

\noindent In \cite{NG14the}, the author showed that if $w_b$ does not approach $\phi$ exponentially fast, where $\phi$ is the positive solution  of the associated ordinary differential equation of equation \eqref{equ:divw1},
\begin{equation}\label{equ:phiODE}
\phi_s = -\frac{\phi}{p-1} + \phi^p + e^{-\frac{ps}{p-1}}h(e^\frac{s}{p-1}\phi) \quad \text{such that} \quad \phi(s) \to \kappa \quad \text{as} \quad s \to + \infty,
\end{equation}
then the solution $u$ of \eqref{equ:problem} would approach an explicit universal profile as follows:
\begin{equation}\label{equ:behU1}
(T-t)^{\frac{1}{p-1}}u(b + z\sqrt{(T-t)|\log(T-t)|}, t) \to f(z) \quad \text{as}\quad t \to T,
\end{equation}
in $L_{loc}^\infty$ and in the case $a > 1$, where
\begin{equation}\label{def:f}
f(z) = \kappa\left(1 + c_p|z|^2\right)^{-\frac{1}{p-1}}, \quad \text{with}\;\; c_p=\frac{p-1}{4p}.
\end{equation}
The goal of this work is to show that the behavior \eqref{equ:behU1} does occur. More precisely, we construct a blow-up solution of equation \eqref{equ:problem} satisfying the behavior described in \eqref{equ:behU1}. This is our main result:
\begin{theo}[\textbf{Existence of a blow-up solution for equation \eqref{equ:problem} with the description of its profile}] \label{theo:1} There exists $T > 0$ such that equation \eqref{equ:problem} has a solution $u(x,t)$ in $\mathbb{R}^n \times [0,T)$ satisfying:\\
$i)$ the solution $u$ blows up in finite-time $T$ at the point $b = 0$,\\
$ii)$
\begin{equation}\label{est:th1}
\left\|(T-t)^\frac{1}{p-1}u(\cdot\sqrt{T-t},t) - f \left(\frac{\cdot}{\sqrt{|\log(T-t)|}} \right) \right\|_{W^{1,\infty}(\Rn)} \leq \frac{C}{|\log (T - t)|^\varrho},
\end{equation}
for all $\varrho \in (0,\nu)$ with $\nu = \min\{a - 1, \frac 12\}$ in the case \eqref{equ:h} and $\nu = \min\{a, \frac 12\}$ in the case \eqref{equ:h1}, $C$ is some positive constant and $f$ is defined in \eqref{def:f}.\\
$iii)$ There exists $u_* \in \mathcal{C}(\Rn \setminus \{0\}, \R)$ such that $u(x,t) \to u_*(x)$ as $t \to T$ uniformly on compact subsets of $\Rn \setminus \{0\}$, where 
$$u_*(x) \sim \left(\frac{8p |\log |x||}{(p-1)^2|x|^2} \right)^{\frac{1}{p-1}}\quad \text{as}\quad x \to 0.$$
\end{theo}

\begin{rema} Note that $i)$ directly follows from $ii)$. Indeed, $ii)$ implies that $u(0,t) \sim \kappa (T-t)^{-\frac{1}{p-1}} \to +\infty$ as $t \to T$, which means that $u$ blows up in finite-time $T$ at the point $0$. From $iii)$, we see that $u$ blows up only at the point $b = 0$.
\end{rema}
\begin{rema} Note that the profile $f$ is the same as in the nonlinear heat equation without the perturbation ($h \equiv 0$), see Bricmont and Kupiainen \cite{BKnon94}, Merle and Zaag \cite{MZdm97}.\\
The estimate \eqref{est:th1} holds in $W^{1,\infty}$ and uniformly in $z \in \mathbb{R}^n$. In the previous work, Ebde and Zaag \cite{EZsema11} gives such a uniform convergence in the case $h$ involving a nonlinear gradient term. In fact, the convergence in $W^{1,\infty}$ comes from a parabolic regularity estimate for equation \eqref{equ:divw1} (see Proposition \ref{prop:regW1in} below). Dealing with the case $h \equiv 0$, Bricmont and Kupiainen \cite{BKnon94}, Merle and Zaag \cite{MZdm97} also give such a uniform convergence but only in $L^\infty(\mathbb{R}^n)$. In most papers, the same kind convergence is proved, but only uniformly on a smaller subsets, $|z| \leq K\sqrt{|\log(T-t)|}$ (see Vel\'azquez \cite{VELcpde92}).
\end{rema}

The proof of Theorem \ref{theo:1} bases on techniques developed by Bricmont and Kupiainen in \cite{BKnon94} and Merle and Zaag in \cite{MZdm97} for the semilinear heat equation
\begin{equation}\label{equ:semiheat}
u_t = \Delta u + |u|^{p-1}u.
\end{equation}
Note that the perturbation term $h$ certainly impacts on the construction of solutions of \eqref{equ:problem} satisfying \eqref{est:th1}. This causes some crucial modifications in \cite{MZdm97} in order to totally control the term $h$. Although these modifications do not affect the general framework developed in \cite{MZdm97}, they lay in 3 crucial places:\\
- We modify the profile around which we study equation \eqref{equ:divw1}, so that we go beyond the order $\frac{1}{s^a}$ generated by the perturbation term. Indeed, for small $a > 0$ and with the same profile as in \cite{MZdm97}, the order $\frac{1}{s^a}$ will become too strong and will not allow us to close our estimates. See Section 2 below, particularly definition \eqref{def:varphi}, which enables us to reach the order $\frac{1}{s^{a + 1}}$.\\
- In order to handle the order $\frac{1}{s^{a+1}}$, we need to modify the definition of the shrinking set near the profile. See Section 3 and particularly Proposition \ref{prop:1} below.\\
- A sharp understanding of the dynamics of the linearized operator of \eqref{equ:divw1} around the profile \eqref{def:varphi}, and which allows to handle the new definition of the shrinking set. See Lemma \eqref{lemm:BKespri} below.\\

\noindent For that reason, we will stress only the main parts of the proof of Theorem \ref{theo:1} and put forward the novelties of our argument. In particular, the proof relies on the understanding of the dynamics of the self-similar version of equation 
\eqref{equ:divw1} around the profile \eqref{def:f}. Following the work by Merle and Zaag \cite{MZdm97}, the proof will be divided into 2 steps:\\
- In the first step, we reduce the problem to a finite-dimensional problem: we will show that it is enough to control a finite-dimensional variable in order to control the solution near the profile. \\
- In the second step, we proceed by contradiction to solve the finite-dimensional problem and conclude using index theory.\\

\noindent We would like to mention that Masmoudi and Zaag \cite{MZjfa08} adapted the method of \cite{MZdm97} for the following Ginzburg-Landau equation:
\begin{equation}\label{equ:comGin}
u_t = (1 + \imath \beta)\Delta u + (1 + \imath \delta)|u|^{p-1}u,
\end{equation}
where $p - \delta^2 - \beta\delta (p+1) > 0$ and $u: \mathbb{R}^n\times[0,T) \to \mathbb{C}$. 
Note that the case $\beta = 0$ and $\delta \in \mathbb{R}$ small has been  studied earlier by Zaag \cite{ZAAihn98}.\\
In \cite{NZpre14}, Nouaili and Zaag successfully used the method of \cite{MZdm97} for the following complex valued semilinear heat equation:
$$u_t = \Delta u + u^2,$$
where $u(t): x \in \Rn \to \mathbb{C}$.\\

\noindent As in \cite{MZdm97}, \cite{ZAAihn98} and \cite{MZjfa08}, it is possible to make the interpretation of the finite-dimensional variable in terms of the blow-up time and the blow-up point. This allows us to derive the stability of the profile $f$ in Theorem \ref{theo:1} with respect to perturbations in the initial data. More precisely, we have the following:
\begin{theo}[\textbf{Stability of the solution constructed in Theorem \ref{theo:1}}] \label{theo:2}
Let us denote by $\hat{u}(x,t)$ the solution constructed in Theorem \ref{theo:1} and by $\hat{T}$ its blow-up time. Then, there exists a neighborhood $\mathcal{V}_0$ of $\hat{u}(x,0)$ in $W^{1,\infty}$ such that for any $u_0 \in \mathcal{V}_0$, equation \eqref{equ:problem} has a unique solution $u(x, t)$ with initial data $u_0$, and $u(x,t)$ blows up in finite time $T(u_0)$ at one single blow-up point $b(u_0)$. Moreover, estimate \eqref{equ:behU1} is satisfied by $u(x-b,t)$ and
$$T(u_0) \to \hat{T}, \quad b(u_0) \to 0 \quad \text{as $u_0 \to \hat{u}_0$ in $W^{1,\infty}(\mathbb{R}^n)$}.$$
\end{theo}
\begin{rema}  We will not give the proof of Theorem \ref{theo:2} because the stability result follows from the reduction to a finite dimensional case as in \cite{MZdm97} with the same proof. Hence, we only prove the reduction and refer to \cite{MZdm97} for the stability. Note that from the parabolic regularity, our stability result holds in the larger space $L^\infty(\Rn)$.
\end{rema}

\section{Formulation of the problem}
As in \cite{MZdm97}, \cite{BKnon94}, we give the proof in one dimension ($n=1$). The proof remains the same for higher dimensions ($n \geq 2$). We would like to find $u_0$ initial data such that the solution $u$ of equation \eqref{equ:problem} blows up in finite time $T$ and satisfies the estimate \eqref{est:th1}. Using similarity variables \eqref{equ:simivariables}, this is equivalent to finding $s_0 > 0$ and $w_0(y)\equiv w(y,s_0)$ such that the solution $w$ of equation \eqref{equ:divw1} with initial data $w_0$ satisfies
\begin{equation*}
\lim_{s \to +\infty}\| w(y,s) - f(\frac{y}{\sqrt{s}})\|_{W^{1,\infty}} = 0, 
\end{equation*}
where $f$ is given in \eqref{def:f}.\\
In order to prove this, we will not linearize equation \eqref{equ:divw1} around $f + \frac{\kappa}{2p s}$ as in \cite{MZdm97}. We will instead introduce 
\begin{equation}\label{def:varphi}
q = w-\varphi, \quad \text{where} \quad \varphi = \frac{\phi(s)}{\kappa}\left(f(\frac{y}{\sqrt{s}}) + \frac{\kappa}{2p s}\right),
\end{equation}
with $\phi$ and $f$ are introduced in \eqref{equ:phiODE} and \eqref{def:f}. Then, the problem is reduced to constructing a function $q$ such that 
\begin{equation*}
\lim_{s \to + \infty}\|q(y,s)\|_{W^{1,\infty}} = 0
\end{equation*}
and $q$ is a solution of the following equation for all $(y,s) \in \mathbb{R} \times [s_0, +\infty)$,
\begin{equation}\label{equ:q}
q_s = (\mathcal{L} + V)q + B(q) + R(y,s) + N(y,s),
\end{equation}
where $\mathcal{L} = \Delta - \frac{y}{2}\cdot \nabla + 1$ and
\begin{align}
V(y,s) &= p \left(\varphi(y,s)^{p-1} - \frac{1}{p-1}\right) + \imath e^{-s}h'(e^\frac{s}{p-1}\varphi),\label{def:V}\\
B(q) &= |\varphi + q|^{p-1}(\varphi + q) - \varphi^p - p\varphi^{p-1}q,\label{def:B}\\
R(y,s) &= -\varphi_s + \Delta \varphi - \frac{y}{2}\cdot \nabla \varphi - \frac{\varphi}{p-1} + \varphi^p + e^{\frac{-ps}{p-1}}h\left(e^{\frac{s}{p-1}}\varphi\right),\label{def:R}\\
N(q,s) &= e^{\frac{-ps}{p-1}}\left[h\left(e^\frac{s}{p-1}(\varphi + q) \right) - h\left(e^\frac{s}{p-1}\varphi \right) - \imath e^\frac{s}{p-1}h'\left(e^\frac{s}{p-1}\varphi \right)q\right],\label{def:N}
\end{align}
with $\imath = 0$ in the case \eqref{equ:h} and $\imath  = 1$ in the case \eqref{equ:h1}.\\

\noindent One can remark that we don't linearize \eqref{equ:divw1} around $\tilde{\varphi} = f(\frac{y}{\sqrt{s}}) + \frac{\kappa}{2p s}$ as in the case of equation \eqref{equ:semiheat} treated in \cite{MZdm97}. In fact, if we do the same, we may obtain some terms like $\frac{1}{s^a}$ coming from the strong perturbation $h$ in equation \eqref{equ:divw1}, and we may not be able to control these terms in the case $a < 3$. To extend the range of $a$, we multiply the factor $\frac{\phi(s)}{\kappa}$ to $\tilde{\varphi}$ in order to go beyond the order $\frac{1}{s^a}$ and reach at the order $\frac{1}{s^{a + 1}}$. Linearizing around $\varphi$ given in \eqref{def:varphi} is a major novelty in our approach. \\

\noindent In following analysis, we will use the following integral form of equation \eqref{equ:q}: for each $s \geq \sigma \geq s_0$:
\begin{equation}\label{for:qint}
q(s) = \mathcal{K}(s,\sigma)q(\sigma) + \int_\sigma^s \mathcal{K}(s, \tau)\left[B(q(\tau)) + R(\tau) + N(q(\tau),\tau)\right] d\tau,
\end{equation}
where $\mathcal{K}$ is the fundamental solution of the linear operator $\mathcal{L} + V$ defined for each $\sigma > 0$ and for each $s \geq \sigma$,
\begin{equation} \label{def:kernel}
\partial_s\mathcal{K}(s,\sigma) = (\mathcal{L} + V)\mathcal{K}(s,\sigma), \quad \mathcal{K}(\sigma, \sigma) = Identity.
\end{equation}

\noindent Since the dynamics of equation \eqref{equ:q} are influenced by the linear part, we first need to recall some properties of the operator $\mathcal{L}$ from Bricmont and Kupiainen \cite{BKnon94}. The operator $\mathcal{L}$ is self-adjoint in $L^2_\rho(\mathbb{R}^n)$, where $L^2_\rho$ is the weighted $L^2$ space associated with the weight $\rho$ defined by 
$$\rho(y) = \left(\frac{1}{4\pi}\right)^{n/2}e^{-\frac{|y|^2}{4}}.$$
Its spectrum is given by 
$$spec(\mathcal{L}) = \{1 - \frac{m}{2},\; m \in \mathbb{N}\},$$
and its eigenfunctions are derived from Hermite polynomials.\\
If $n = 1$, the eigenfunction corresponding to $1 - \frac{m}{2}$ is
\begin{equation}\label{equ:eigenfu1}
h_m(y) = \sum_{k= 0}^{\left[\frac{m}{2}\right]} \frac{m!}{k!(m - 2k)!}(-1)^ky^{m - 2k}. 
\end{equation}
We also denote $k_m(y) =  \frac{h_m(y)}{\|h_m(y)\|_{L^2_\rho}^2}$.\\
If $n \geq 2$, we write the spectrum of $\mathcal{L}$ as $spec(\mathcal{L}) = \{ 1 - \frac{|m|}{2},\; |m| = m_1 + \dots + m_n, \;(m_1,\dots, m_n) \in \mathbb{N}^n\}$. Given $m = (m_1, \dots, m_n) \in \mathbb{N}^n$, the eigenfunction corresponding to $1 - \frac{|m|}{2}$ is 
\begin{equation}\label{equ:eigenfu}
H_m(y) = h_{m_1}(y_1)\dots h_{m_n}(y_n), \quad \text{where $h_{m}$ is defined in \eqref{equ:eigenfu1}.}
\end{equation}
The potential $V(y,s)$ has two fundamental properties:\\
$i)$ $V(\cdot, s) \to 0$ in $L^2_{\rho}$ as $s \to +\infty$. In particular, the effect of $V$ on the bounded sets or in the "blow-up" region ($|y| \leq K\sqrt{s}$) is regarded as a perturbation of the effect of $\mathcal{L}$.\\
$ii)$ outside of the "blow-up" region, we have the following property: for all $\epsilon > 0$, there exist $C_\epsilon > 0$ and $s_\epsilon$ such that 
\begin{equation}\label{equ:asymV}
\sup_{s \geq s_\epsilon, |y| \geq C_\epsilon \sqrt{s}} |V(y,s) - (-\frac{p}{p-1})| \leq \epsilon.
\end{equation}
This means that $\mathcal{L} + V$ behaves like $\mathcal{L} - \frac{p}{p-1}$ in the region $|y| \geq K\sqrt{s}$. Because $1$ is the biggest eigenvalue of $\mathcal{L}$, the operator $\mathcal{L} - \frac{p}{p-1}$ has purely negative spectrum. Therefore, the control of $q(y,s)$ in $L^\infty$ outside of the "blow-up" region will be done without difficulties. \\
Since the behavior of $V$ inside and outside of the "blow-up" region are different, let us decompose $q$ as following: Let $\chi_0 \in \mathcal{C}_0^\infty([0,+\infty))$ with $supp(\chi_0) \subset [0,2]$ and $\chi_0 \equiv 1$ on $[0,1]$. We define 
\begin{equation}\label{def:chi}
\chi(y,s)= \chi_0(\frac{|y|}{K\sqrt{s}}),
\end{equation}
where $K > 0$ to be fixed large enough, and write 
\begin{equation}\label{de:qbqe}
q(y,s) = q_b(y,s) + q_e(y,s),
\end{equation}
where $q_b(y,s) = \chi(y,s) q(y,s)$ and $q_e(y,s) = (1 - \chi(y,s))q(y,s)$. Note that $supp (q_b(s)) \subset \mathbf{B}(0,2K\sqrt{s})$ and $supp (q_e(s)) \subset \mathbb{R}\setminus\mathbf{B}(0,K\sqrt{s})$.\\
In order to control $q_b$, we expand it with respect to the spectrum of $\mL$ in $L^2_\rho$. More precisely, we write $q$ into 5 components as follows:
\begin{equation}\label{def:decomq}
q(y,s) = \sum_{m=0}^2q_m(s)h_m(y) + q_-(y,s) + q_e(y,s),
\end{equation}
where $q_m, q_-$ are coordinates of $q_b$ (not of $q$), namely that $q_m$ is the projection of $q_b$ in $h_m$ and $q_- = P_-(q_b)$ with $P_-$ being the projector on the negative subspace of $\mathcal{L}$.

\section{Proof of the existence of a blow-up solution with the given blow-up profile}
In this section, we use the framework developed in \cite{MZdm97} in order to prove Theorem \ref{theo:1}. We proceed in 4 steps:\\
- In the first step, we define a shrinking set $V_A(s)$ and translate our goal of making $q(s)$ go to $0$ in $L^\infty(\mathbb{R})$ in terms of belonging to $V_A(s)$. We also exhibit a two parameter initial data family for equation \eqref{equ:q} whose coordinates are very small (with respect to the requirements of $V_A(s)$), except the two first $q_0$ and $q_1$. Note that the set $V_A(s)$ is different from the corresponding one in \cite{MZdm97}, and this makes the second major novelty of our work, in addition to the modification of the profile in \eqref{def:varphi}. \\
- In the second step, using the spectral properties of equation \eqref{equ:q}, we reduce our
goal from the control of $q(s)$ (an infinite dimensional variable) in $V_A(s)$ to the control of its two first components $(q_0(s), q_1(s))$ (a two-dimensional variable) in $\left[ -\frac{A}{s^{1 + \nu}}, \frac{A}{s^{1 + \nu}}\right]^2$ with $\nu > 0$.\\
- In the third step, we solve the local in time Cauchy problem for equation \eqref{equ:q}.\\
- In the last step, we solve the finite dimensional problem using index theory
and conclude the proof of Theorem \ref{theo:1}.\\

In what follows, the constant $C$ denotes a universal one independent of variables, only depending upon constants of the problems such as $a$, $p$, $M$, $\mu$ and $K$ in \eqref{def:chi}.

\subsection{Definition of a shrinking set $V_A(s)$ and preparation of initial data}
Let first introduce the following proposition:
\begin{prop}[\textbf{A shrinking set to zero}] \label{prop:1} Let $\nu = \min\{a - 1, \frac{1}{2}\}$ in the case \eqref{equ:h} and $\nu = \min\{a, \frac{1}{2}\}$ in the case \eqref{equ:h1}, we fix $\varrho \in (0,\nu)$. For each $A > 0$, for each $s > 0$, we define $\hat{V}_A(s) = \left[-\frac{A}{s^{1 + \nu}}, \frac{A}{s^{1 + \nu}} \right]^2 \subset \mathbb{R}^2$, and $V_A(s)$ as being the set of all functions $g$ in $L^\infty$ such that:
$$m = 0,1,\;\;|g_m(s)| \leq \frac{A}{s^{1 + \nu}}, \quad |g_2(s)| \leq \frac{A^2}{s^{1+\nu}},$$
$$\forall y \in \mathbb{R},\quad  |g_-(y,s)| \leq \frac{A}{s^{3/2+\varrho}}(1 + |y|^3), \quad \|g_e(s)\|_{L^\infty} \leq \frac{A^2}{s^\varrho},$$
where $g_m, g_-$ and $g_e$ are defined in \eqref{def:decomq}. Then we have for all $s \geq e$ and $g \in V_A(s)$, 
\begin{equation}\label{iq:VAbe}
\forall y\in\R,\;\; |g(y,s)| \leq \frac{CA^2}{s^{3/2 + \varrho}}(1 + |y|^3)  + \frac{CA^2}{s^{1 + \nu}}(1 + |y|^2) \quad \text{and}\quad  \|g(s)\|_{L^\infty} \leq \frac{CA^2}{s^\varrho}.
\end{equation}
\end{prop}
\begin{proof} From the definition of $V_A(s)$ and the fact that $\left|\frac{1 - \chi(y,s)}{1 + |y|^3}\right| \leq \frac{C}{s^{3/2}}$, the conclusion of Proposition \ref{prop:1} simply follows.
\end{proof}
Initial data (at time $s_0 = -\log T$) for the equation \eqref{equ:q} will depend on two real parameters $d_0$ and $d_1$ as given in the following proposition:
\begin{lemm}[\textbf{Decomposition of initial data on the different components)}] \label{lemm:decomdata} For each $A > 1$, there exists $\delta_1(A) > 0$ such that for all $s_0 \geq \delta_1(A)$: If we consider the following function as initial data for equation \eqref{equ:q}:
\begin{equation}\label{def:intq0}
q_{d_0,d_1}(y,s_0) = \frac{\phi(s_0)}{\kappa}\left(f^p(z)(d_0 + d_1 z) - \frac{\kappa}{2p s_0}\right),
\end{equation}
where $z = \frac{y}{\sqrt{s_0}}$, $f$ and $\phi$ are defined in \eqref{def:f} and \eqref{equ:phiODE}, then\\
$i)$ There exists a constant $C = C(p) > 0$ such that the components of $q_{d_0,d_1}(s_0)$ (or $q(s_0)$ for short) satisfy:
\begin{align}
q_0(s_0) &= d_0a_0(s_0) + b_0(s_0), \quad \text{with} \quad a_0(s_0) \sim C, \quad |b_0(s_0)| \leq \frac{C}{s_0},\label{equ:q_0init}\\
q_1(s_0) &= d_1a_1(s_0) + b_1(s_0), \quad \text{with} \quad a_1(s_0) \sim \frac{C}{\sqrt{s_0}}, \quad |b_1(s_0)| \leq \frac{C}{s_0^2},\label{equ:q_1init}
\end{align}
and 
\begin{align*}
&|q_2(s_0)| \leq \frac{C |d_0|}{s_0} + Ce^{-s_0}, \quad |q_-(y,s_0)| \leq \left(\frac{C|d_0|}{s_0} + \frac{C|d_1|}{s_0\sqrt{s_0}} \right)(1 + |y|^3),\\ 
&\|q_e(s_0)\|_{L^\infty} \leq C|d_0| + \frac{C|d_1|}{\sqrt{s_0}}, \quad \|\nabla q(s_0)\|_{L^\infty} \leq \frac{C(|d_0| + |d_1|)}{\sqrt{s_0}}.
\end{align*}
$ii)$ For each $A > 0$, if $(d_0,d_1)$ is chosen so that $(q_0,q_1)(s_0) \in \hat{V}_A(s_0)$, then 
\begin{align*}
&|d_0| + |d_1| \leq \frac{C}{s_0},\\
&|q_2(s_0)| \leq \frac{C}{s_0^2}, \;\;\left\|\frac{q_-(y,s_0)}{1 + |y|^3} \right\|_{L^\infty} \leq \frac{C}{s_0^2}, \;\; \|q_e(s_0)\|_{L^\infty} \leq \frac{C}{s_0},\\
&q(s_0) \in V_A(s_0), \;\; \|\nabla q(s_0)\|_{L^\infty} \leq \frac{C}{s_0\sqrt{s_0}},
\end{align*}
where the statement $q(s_0) \in V_A(s_0)$ holds with "strict inequalities", except for $(q_0,q_1)(s_0)$, in the sense that 
$$m = 0,1,\;\;|q_m(s)| \leq \frac{A}{s^{1 + \nu}}, \quad |q_2(s)| < \frac{A^2}{s^{1+\nu}},$$
$$\forall y \in \mathbb{R},\quad  |q_-(y,s)| < \frac{A}{s^{3/2+\varrho}}(1 + |y|^3), \quad \|q_e(s)\|_{L^\infty} < \frac{A^2}{s^\varrho}.$$
$iii)$ There exists a rectangle $\mathcal{D}_{s_0} \subset \left[-\frac{C}{s_0},\frac{C}{s_0} \right]^2$ such that the mapping $(d_0,d_1) \mapsto (q_0, q_1)(s_0)$ is linear and one to one from $\mathcal{D}_{s_0}$ onto $\left[-\frac{A}{s_0^{1 + \nu}}, \frac{A}{s_0^{1 + \nu}} \right]^2$ and maps $\partial \mathcal{D}_{s_0}$ into $\partial \left[-\frac{A}{s_0^{1 + \nu}}, \frac{A}{s_0^{1 + \nu}} \right]^2$. Moreover, it is of degree one on the boundary and the following equivalence holds:
$$q(s_0) \in V_A(s_0) \;\; \text{if and only if}\;\; (d_0,d_1) \in \mathcal{D}_{s_0}.$$
\end{lemm}
\begin{proof} $i)$ Since we have the similar expression of initial data \eqref{def:intq0} as in \cite{MZdm97}, we refer the reader to Lemma 3.5 of \cite{MZdm97}, except for the bound on $\|\nabla q(s_0)\|_{L^\infty}$. Note that although $i)$ is not stated explicitly in Lemma 3.5 of \cite{MZdm97}, they are clearly written in its proof. For $\|\nabla q(s_0)\|_{L^\infty}$, we use \eqref{def:intq0} and the fact that $f'(z) = -\frac{p-1}{2p}zf^p(z)$, $f^p(z), zf^{p-1}(z)$ and $z^2f^{p-1}(z)$ are in $L^\infty(\R)$ to derive
\begin{align*}
|\nabla q(y,s_0)| &\leq \left|\frac{\phi(s_0)}{\kappa}\right| \left|\frac{f^p(z)}{\sqrt{s_0}}\left( pd_0zf^{p-1}(z)  + d_1  +pd_1z^2f^{p-1}(z)\right)\right|\\
&\leq \frac{C}{\sqrt{s_0}}(|d_0| + |d_1|).
\end{align*}
$ii)$ We see from \eqref{equ:q_0init} and \eqref{equ:q_1init} that if $(d_0,d_1)$ is chosen so that $(q_0,q_1)(s_0) \in \left[-\frac{A}{s_0^{1 + \nu}}, \frac{A}{s_0^{1 + \nu}} \right]^2$, then $|d_0|$ and $|d_1|$ are bounded by $\frac{C}{s_0}$. Substituting these bounds into the estimates stated in $i)$, we immediately derive $ii)$.\\
$iii)$ It follows from \eqref{equ:q_0init} and \eqref{equ:q_1init}, part $ii)$ and the definition of $V_A$ given in Proposition \ref{prop:1}. This ends the proof of Lemma \ref{lemm:decomdata}.
\end{proof}
As stated in Theorem \ref{theo:1}, the convergence holds in $W^{1,\infty}(\R)$, we need the following parabolic regularity estimate for equation \eqref{equ:q}, with $q(s_0)$ given by \eqref{for:qint} and $q(s) \in V_A(s)$. More precisely, we have the following:
\begin{prop}\label{prop:regW1in}  For each $A \geq 1$, there exists $\delta_2(A) > 0$ such that for all $s_0 \geq \delta_2(A)$: if $q(s)$ is a solution of equation \eqref{equ:q} on $[s_0,s_1]$ with initial data at $s = s_0$, $q_{d_0,d_1}(s_0)$ given in \eqref{for:qint} where $(d_0,d_1) \in \mathcal{D}_{s_0}$, assume in addition that $q(s) \in V_A(s)$ for $s \in [s_0,s_1]$, then 
$$\|\nabla q(s)\|_{L^\infty} \leq \frac{CA^2}{s^\varrho}, \quad \forall s \in [s_0,s_1],$$
for some positive constant $C$.
\end{prop}
\begin{proof} The proof is the same as Proposition 3.3 of \cite{EZsema11}. We would like to mention that the proof bases on a Gronwall's argument and the following properties of the kernel $e^{\theta \mL}$ defined in \eqref{for:kernalL}:
$$\forall g \in L^\infty, \;\; \|\nabla (e^{\theta \mL}g)\|_{L^\infty} \leq \frac{Ce^{\theta/2}\|g\|_{L^\infty}}{\sqrt{1 - e^{-\theta}}},$$
and 
$$\forall f \in W^{1,\infty},\;\; \|\nabla (e^{\theta \mL}f)\|_{L^\infty} \leq Ce^{\theta/2}\|\nabla f\|_{L^\infty}.$$
Although the definition of $V_A$ is slightly different from the one defined in \cite{EZsema11}, the reader will have absolutely no difficulty to adapt their proof to the new situation. For that reason, we refer the reader to \cite{EZsema11} for details of the proof.
\end{proof}

\subsection{Reduction to a finite dimensional problem}
We are going to the crucial step of the proof of Theorem \ref{theo:1}. In this step, we will show that through a priori estimates, the control of $q(s)$ in $V_A$ reduces to the control of $(q_0,q_1)(s)$ in $\hat{V}_A(s)$.  As presented in \cite{MZdm97} (see also\cite{ZAAihn98},\cite{MZjfa08}), we would like to emphasize that this step make the heart of the contribution. Even more, here lays another major contribution of ours, in the sense that we understand better the dynamics of the fundamental solution $\K(s,\sigma)$ defined in \eqref{def:kernel}. Our sharper estimates are given in Lemma \ref{lemm:BKespri} below. In fact all that we do is to rewrite the corresponding estimates of Bricmont and Kupiainen \cite{BKnon94} without taking  into account the particular form of the shrinking set they used. Furthermore, because of the difference in the definition \eqref{def:varphi} of $\varphi$ and the difference in the definition of $V_A$, the proof is far from being an adaptation of the proof written in \cite{MZdm97}. We therefore need some involved arguments to control the components of $q$ and conclude the reduction to a finite dimensional problem. \\

\noindent We mainly claim the following:
\begin{prop}[\textbf{Control of $q(s)$ by $(q_0,q_1)(s)$ in $V_A(s)$}] \label{prop:redu} There exist $A_3 > 0$ such that for each $A \geq A_3$, there exists $\delta_3(A) > 0$ such that for each $s_0 \geq \delta_3(A)$, we have the following properties:\\
- if $(d_0, d_1)$ is chosen so that $(q_0,q_1)(s_0) \in \hat{V}_A(s_0)$, and \\
- if for all $s \in [s_0,s_1]$, $q(s) \in V_A(s)$ and $q(s_1) \in \partial V_A(s_1)$ for some $s_1 \geq s_0$, then:\\
$i)\;$ \textbf{(Reduction to a finite dimensional problem)} $\;\; (q_0, q_1)(s_1) \in \partial \hat{V}_A(s_1)$,\\
$ii)$ \textbf{(Transversality)} there exists $\eta_0 > 0$ such that for all $\eta \in (0, \eta_0)$, $\; (q_0, q_1)(s_1 + \eta) \not\in \partial \hat{V}_A(s_1 + \eta)$ (hence, $q(s_1 + \eta) \not\in V_A(s_1 + \eta)$).
\end{prop}
The proof follows the general ideas of \cite{MZdm97} and we proceed in three steps:\\
- Step 1: we give a priori estimates  on $q(s)$ in $V_A(s)$: assume that for given $A > 0$ lager, $\lambda > 0$ and an initial time $s_0 \geq \sigma_2(A,\lambda) \geq 1$, we have $q(s) \in V_A(s)$ for each $s \in [\tau, \tau + \lambda]$ where $\tau \geq s_0$, then using the integral form \eqref{for:qint} of $q(s)$, we derive new bounds on $q_2(s), q_-(s)$ and $q_e(s)$ for $s \in [\tau, \tau + \lambda]$.\\
- Step 2: we show that these new bounds are better than those defining $V_A(s)$. It then remains to control $q_0(s)$ and $q_1(s)$. This means that the problem is reduced to the control  of a two dimensional variable $(q_0, q_1)(s)$ and we then conclude $i)$ of Proposition \ref{prop:redu}.\\
- Step 3: we use dynamics of $(q_0, q_1)(s)$ to show its transversality on $\partial V_A(s)$, which corresponds to part $ii)$ of Proposition \ref{prop:redu}.

\subsubsection*{Step 1: A priori estimates on $q(s)$ in $V_A(s)$}
As indicated above, the derivation of the new bounds on the components of $q(s)$ bases on the integral formula \eqref{for:qint}. It is clear to see the strong influence of the kernel $\K$ in this formula. Therefore, it is convenient to give the following result from Bricmont and Kupiainen in \cite{BKnon94} which gives the dynamics of the linear operator $\mathcal{L} + V$:
\begin{lemm}[\textbf{Bricmont and Kupiainen \cite{BKnon94}}] \label{lemm:BKespri} For all $\lambda > 0$, there exists $\sigma_0 = \sigma_0(\lambda)$ such that if $\sigma \geq \sigma_0 \geq 1$ and $\psi(\sigma)$ satisfies 
\begin{equation}\label{equ:boundpsisigma}
\sum_{m=0}^2|\psi_m(\sigma)| + \left\|\frac{\psi_-(y,\sigma)}{1+|y|^3}\right\|_{L^\infty} + \|\psi_e(\sigma)\|_{L^\infty} < +\infty,
\end{equation}
then, $\theta(s) = \mathcal{K}(s,\sigma)\psi(\sigma)$ satisfies for all $s \in [\sigma, \sigma + \lambda]$,
\begin{align}
|\theta_2(s)| &\leq \left(\frac{\sigma}{s}\right)^2 |\psi_2(\sigma)|+\frac{C(s-\sigma)}{s} \left(\sum_{l=0}^2|\psi_l(\sigma)| + \left\|\frac{\psi_-(y,\sigma)}{1+|y|^3} \right\|_{L^\infty}\right)\nonumber\\
& \qquad + C(s-\sigma)e^{-s/2}\|\psi_e(\sigma)\|_{L^\infty},\label{equ:boundThe_2}\\
\left\|\frac{\theta_-(y,s)}{1+|y|^3}\right\|_{L^\infty} &\leq \frac{Ce^{s - \sigma}\left((s - \sigma)^2 + 1\right)}{s}\left(|\psi_0(\sigma)| + |\psi_1(\sigma)|+\sqrt{s}|\psi_2(\sigma)|\right)\nonumber\\
&\qquad + C e^{-\frac{(s - \sigma)}{2}} \left\|\frac{\psi_-(y,\sigma)}{1 + |y|^3}\right\|_{L^\infty} + \frac{Ce^{-(s - \sigma)^2}}{s^{3/2}}\|\psi_e(\sigma)\|_{L^\infty},\label{equ:boundThe_ne}\\
\|\theta_e(s)\|_{L^\infty} &\leq Ce^{s-\sigma}\left(\sum_{l=0}^2 s^{l/2}|\psi_l(\sigma)| + s^{3/2}\left\|\frac{\psi_-(y,\sigma)}{1 + |y|^3}\right\|_{L^\infty} \right)\nonumber\\
&\qquad + C e^{-\frac{(s-\sigma)}{p}}\|\psi_e(\sigma)\|_{L^\infty}.\label{equ:boundThe_e}
\end{align}
where $C = C(\lambda,K)>0$ ($K$ is given in \eqref{def:chi}), $\psi_m,\psi_-,\psi_e$ and $\theta_m, \theta_-, \theta_e$ are defined by \eqref{de:qbqe} and \eqref{def:decomq}.
\end{lemm}
\begin{rema} In view of the formula \eqref{for:qint}, we see that Lemma \ref{lemm:BKespri} will play an important role in deriving the new bounds on the components of $q(s)$ and making our proof simpler. This means that, given bounds on the components of $q(\sigma), B(q(\tau)), R(\tau), N(q(\tau),\tau)$, we directly apply Lemma \ref{lemm:BKespri} with $\K(s, \sigma)$ replaced by $\K(s,\tau)$ and then integrate over $\tau$ to obtain estimates on the components of $q$.
\end{rema}
\begin{proof} Let us mention that Lemma \ref{lemm:BKespri} relays mainly on the understanding of the behavior of the kernel $\mathcal{K}(s,\sigma)$. The proof is essentially the same as in \cite{BKnon94}, but those estimates did not present explicitly the dependence on all the components of $\psi(\sigma)$ which is less convenient for our analysis below. Because the proof is long and technical, we leave it to Appendix \ref{ap:2a}.
\end{proof}

We now assume that for each $\lambda > 0$, for each $s \in [\sigma, \sigma + \lambda]$, we have $q(s) \in V_A(s)$ with $\sigma \geq s_0$. Applying Lemma \ref{lemm:BKespri}, we get new bounds on all terms in the right hand side of \eqref{for:qint}, and then on $q$. More precisely, we claim the following:
\begin{lemm}\label{lemm:estallterms} There exists $A_2 > 0$ such that for each $A \geq A_2$, $\lambda^* > 0$, there exists $\sigma_2(A,\lambda^*) > 0$ with the following property: for all $s_0 \geq \sigma_2(A, \lambda^*)$, for all $\lambda \leq \lambda^*$, assume that for all $s \in [\sigma, \sigma + \lambda]$, $q(s) \in V_A(s)$ with $\sigma \geq s_0$, then we have for all $s \in [\sigma, \sigma + \lambda]$,\\
$i)\;$ \textbf{(linear term)}
\begin{align*}
|\alpha_2(s)| &\leq \left(\frac{\sigma}{s} \right)^{1 -\nu}\frac{A^2}{s^{1 + \nu}} + \frac{CA^2(s - \sigma)}{s^{2 + \nu}},\\
\left\|\frac{\alpha_-(y,s)}{1 + |y|^3}\right\|_{L^\infty} &\leq \frac{C}{s^{3/2 + \varrho}} + \frac{C}{s^{3/2 + \varrho}}\left(Ae^{-\frac{s - \sigma}{2}} + A^2e^{-(s - \sigma)^2}\right) ,\\
\|\alpha_e(s)\|_{L^\infty} &\leq \frac{C}{s^\varrho} + \frac{C}{s^\varrho}\left(A e^{s - \sigma} + A^2e^{-\frac{s-\sigma}{p}} \right),
\end{align*}
where 
$$\K(s,\sigma)q(\sigma) = \alpha(y,s) = \sum_{m=0}^2\alpha_m(s)h_m(y) + \alpha_-(y,s) + \alpha_e(y,s).$$
If $\sigma = s_0$, we assume in addition that $(d_0,d_1)$ is chosen so that $(q_0,q_1)(s_0) \in \hat{V}_A(s_0)$. Then for all $s \in [s_0, s_0 + \lambda]$, we have
\begin{align*}
&|\alpha_2(s)| \leq \frac{C}{s^2},\quad \left\|\frac{\alpha_-(y,s)}{1 + |y|^3}\right\|_{L^\infty} \leq  \frac{C}{s^2},\quad \|\alpha_e(s)\|_{L^\infty} \leq \frac{Ce^{s - s_0}}{\sqrt{s}}.
\end{align*}
$ii)\,$ \textbf{(remaining terms)}
\begin{align*}
&|\beta_2(s)| \leq \frac{C(s - \sigma)}{s^{2 + \nu}},\quad
\left\|\frac{\beta_-(y,s)}{1 + |y|^3}\right\|_{L^\infty} \leq \frac{C}{s^{3/2 + \varrho}}, \quad \|\beta_e(s)\|_{L^\infty} \leq \frac{C}{s^{\varrho}},
\end{align*}
where 
\begin{align*}
&\int_\sigma^s\K(s,\tau)\left[B(q(\tau)) + R(\tau) + N(q(\tau),\tau) \right]d\tau \\
&\qquad \qquad = \beta(y,s) = \sum_{m=0}^2\beta_m(s)h_m(y) + \beta_-(y,s) + \beta_e(y,s).
\end{align*}
\end{lemm}
\begin{proof} $i)$ It immediately follows from the definition of $V_A(\sigma)$ and Lemma \ref{lemm:BKespri}. Note that in the case $\sigma = s_0$, we use in addition part $ii)$ of Lemma \ref{lemm:decomdata} to have the conclusion. For part $ii)$, all what we need to do is to find the estimates on the components of different terms appearing in equation \eqref{equ:q}, then we use Lemma \ref{lemm:BKespri} and the linearity to have the conclusion. We claim the following:
\begin{lemm}\label{lemm:estonBRN} We have the following properties:\\
$i)$ \textbf{(Estimates on $B(q)$)} For all $A > 0$, there exists $\sigma_3(A)$ such that for all $\tau \geq \sigma_3(A)$, $q(\tau) \in V_A(\tau)$ implies\\
\begin{equation}\label{est:alltermsofB}
m = 0,1,2, \; |B_m(\tau)| \leq \frac{CA^4}{\tau^{2 + 2\nu}}, \; \left|\frac{B_-(y,\tau)}{1+|y|^3}\right| \leq \frac{CA^4}{\tau^{3/2 + 2\varrho}},\; \|B_e(\tau)\|_{L^\infty} \leq \frac{CA^{2p'}}{\tau^{\varrho p'}},
\end{equation}
where $p' = \min \{p,2\}$.\\
$ii)$ \textbf{(Estimates on $R$)} There exists $\sigma_4 > 0$ such that for all $\tau \geq \sigma_4$,\\
\begin{align}
m = 0,1, \;&|R_m(\tau)| \leq \frac{C}{\tau^2},\quad |R_2(\tau)| \leq \frac{C}{\tau^{2 + \nu}},\nonumber\\
\text{and} \;\; &\left\|\frac{R_-(y,\tau)}{1 + |y|^3} \right\|_{L^\infty} \leq \frac{C}{\tau^2}, \quad \|R_e(\tau)\|_{L^\infty} \leq \frac{C}{\tau^\nu}.&\label{equ:estalltermsonR}
\end{align}
$iii)$ \textbf{(Estimates on $N(q,\tau)$)} For all $A > 0$, there exists $\sigma_5(A)$ such that for all $\tau \geq \sigma_5(A)$, $q(\tau) \in V_A(\tau)$ implies\\
\begin{equation} \label{equ:estalltermsofN}
m=0,1,2,\;|N_m(\tau)| \leq \frac{CA^4}{\tau^{2 + 2\nu}},\; \left\|\frac{N_-(y,\tau)}{1 + |y|^3}\right\|_{L^\infty}  \leq \frac{CA^4}{\tau^{2 + 2\varrho}}, \; \|N_e(\tau)\|_{L^\infty} \leq \frac{CA^4}{\tau^{2\varrho}}.
\end{equation}
\end{lemm}
\begin{proof} Since the proof is technical, we leave it to Appendix \ref{ap:lemmEstAll}.
\end{proof}
\vspace*{0.3cm}
\noindent Substituting the estimates stated in Lemma \ref{lemm:estonBRN} into Lemma \ref{lemm:BKespri}, then integrating over $[\sigma,s]$ with respect to $\tau$, and taking $\sigma_2(A,\lambda^*) \geq \max \{\sigma_3,\sigma_4,\sigma_5\}$ such that 
$$\forall s \geq \sigma_2(A,\lambda^*),\;\;(A^4+1)e^{\lambda^*}((\lambda^*)^3 + 1)\left(s^{-\varrho(p' - 1)} + s^{-(\nu - \varrho)} \right) \leq 1,$$
with $p' = \min\{p,2\}$, we have the conclusion. This ends the proof of Lemma \ref{lemm:estallterms}.
\end{proof}

Thanks to Lemma \ref{lemm:estonBRN}, we obtain the following equations satisfied by the expanding modes:
\begin{lemm}[\textbf{ODE satisfied by the expanding modes}] \label{lemm:eqm12} For all $A > 0$, there exists $\sigma_6(A)$ such that for all $s \geq \sigma_6(A)$, $q(s)\in V_A(s)$ implies that for all $s \geq \sigma_6(A)$,
\begin{equation}\label{equ:eqm12}
m = 0,1,\quad \left|q_m'(s) - (1 - \frac{m}{2})q_m(s) \right|\leq \frac{C}{s^{3/2 + \nu}},
\end{equation}
and 
\begin{equation}\label{equ:odeq2}
\left|q_2'(s) + \frac 2s\;q_2(s) \right|\leq \frac{C}{s^{2 + \nu}}.
\end{equation}
\end{lemm}
\begin{proof} The proof is very close to that in \cite{MZdm97}. We therefore give the sketch of the proof. By the definition \eqref{def:decomq}, we write
$$m = 0, 1, 2, \;\;\frac{d q_m(s)}{ds} = \int  \frac{\partial \chi(y,s)}{\partial s} q(s) k_m \rho dy + \int \chi(y,s) \frac{\partial q(s)}{\partial s} k_m \rho dy := I + II.$$
Since the support of $\frac{\partial \chi(y,s)}{\partial s}$ is the set $K\sqrt{s} \leq |y| \leq 2K\sqrt{s}$ (see \eqref{def:chi}), using the fact that $\|q(s)\|_{L^\infty} \leq \frac{C A^2}{s^\varrho}$ (see \eqref{iq:VAbe}), we obtain
$$|I| \leq \int \left|\frac{\partial \chi(y,s)}{\partial s}\right| |q(s)| |k_m| \rho dy \leq CA^2e^{-s}s^{-\varrho},$$
for $K$ large enough.\\
For $II$, we have by equation \eqref{equ:q},
\begin{align*}
II &= \int \chi(y,s) \mL q(s)k_m \rho dy + \int \chi(y,s) V(s)q(s)k_m \rho dy\\
&+ \int \chi(y,s) \left[B(q(s)) + R(s) + N(q(s),s)\right]k_m \rho dy := IIa + IIb + IIc.
\end{align*}
Since $\mL$ is self-adjoint on $L^2_\rho$ and $\mL(\chi(y,s)k_m) = (1 - \frac{m}{2})\chi(y,s) k_m + \frac{\partial ^2\chi(y,s)}{\partial s^2}k_m + \frac{\partial \chi(y,s)}{\partial s}(2 \frac{\partial k_m}{dy} - \frac{y}{2}k_m)$, we obtain 
$$IIa = \int \mL(\chi(y,s) k_m)q(s)\rho dy = (1 - \frac{m}{2})q_m(s) + \mathcal{O}(CA^2e^{-s}),$$
where $\mathcal{O}(r)$ stands for quantity whose absolute value is bounded precisely by $r$ and not $Cr$. \\
Recalling from part $c)$ of Lemma \ref{ap:lemmEstV} that $|V(y,s)| \leq \frac{C}{s}(1 + |y|^2)$ and from \eqref{iq:VAbe} that $|q(y,s)| \leq \frac{CA^2}{s^{1 + \nu}}(1 + |y|^3)$, we derive
$$m = 0, 1, \;\;|IIb| \leq \frac{CA^2}{s^{2 + \nu}}\int(1 + |y|^5)|k_m| \rho dy \leq \frac{CA^2}{s^{2 + \nu}}.$$
For $m = 2$, using the second estimate in part $c)$ of Lemma \ref{ap:lemmEstV}, namely that $V(y,s) = -\frac{h_2(y)}{4s} + \mathcal{O}\left(\frac{C(1+|y|^4)}{s^{1 + \bar{a}}}\right)$ with $\bar{a} = \min\{a-1,a\}$ in the case \eqref{equ:h} and $\bar{a} = \min\{a,1\}$ in the case \eqref{equ:h1}, simultaneously noting that $\int h_2^2 \rho dy = 8$, $\int h_2^3 \rho dy = 64$ and $2 + \bar{a} + \nu \geq 2 + 2\nu$, we obtain
$$m = 2, \;\; IIb = -\frac{2}{s}\;q_2(s) + \mathcal{O}\left(\frac{CA^2}{s^{2 + 2\nu}}\right).$$
The bound for $IIc$ already obtained from \eqref{est:alltermsofB}, \eqref{equ:estalltermsonR} and \eqref{equ:estalltermsofN}. Adding all these bounds and taking $\sigma_6(A)$ large enough such that for all $s \geq \sigma_6(A)$, $A^4s^{-\nu} + A^2s^{2 + \nu}e^{-s} \leq 1$, we then have the conclusion. This ends the proof of Lemma \ref{lemm:eqm12}.
\end{proof}
\subsubsection*{Step 2: Deriving conclusion $i)$ of Proposition \ref{prop:redu}}
Here we use Lemma \ref{lemm:estallterms} in order to derive conclusion $i)$ of Proposition \ref{prop:redu}. Indeed, from equation \eqref{for:qint} and Lemma \ref{lemm:estallterms}, we derive new bounds on $|q_2(s)|$,  $\left\|\frac{q_-(y,s)}{1 + |y|^3} \right\|_{L^\infty}$ and $\|q_e(s)\|_{L^\infty}$, assuming that for all $s \in [\sigma, \sigma +\lambda]$, $q(s) \in V_A(s)$, for $\lambda \leq \lambda^*$ and $\sigma \geq s_0 \geq \sigma_1(A,\lambda^*)$ ($\sigma_1$ is given in Lemma \ref{lemm:estallterms}). The key estimate is to show that for $s = \sigma + \lambda$ (or $s \in [\sigma,\sigma +\lambda]$ if $\sigma = s_0$), these bounds are better than those defining $V_A(s)$, provided that $\lambda \leq \lambda^*(A)$. More precisely, we claim that following proposition which directly follows $i)$ of Proposition \ref{prop:redu}:
\begin{prop}[\textbf{Control of $q(s)$ by $(q_0,q_1)(s)$ in $V_A(s)$}] \label{prop:redu1} There exist $A_4 > 1$ such that for each $A \geq A_4$, there exists $\delta_4(A) > 0$ such that for each $s_0 \geq \delta_4(A)$, we have the following properties:\\
- if $(d_0, d_1)$ is chosen so that $(q_0,q_1)(s_0) \in \hat{V}_A(s_0)$, and \\
- if for all $s \in [s_0,s_1]$, $q(s) \in V_A(s)$ for some $s_1 \geq s_0$, then: for all $s \in [s_0,s_1]$,
\begin{equation}\label{equ:redu1}
|q_2(s)| < \frac{A^2}{s^{1 + \nu}},\quad \left\|\frac{q_-(y,s)}{1 + |y|^3} \right\|_{L^\infty}\leq \frac{A}{2s^{3/2 + \varrho}}, \quad \|q_e(s)\|_{L^\infty} \leq \frac{A^2}{2s^\varrho}.
\end{equation}
\end{prop}
\noindent Let us now derive the conclusion $i)$ of Proposition \ref{prop:redu} from Proposition \ref{prop:redu1}, and we then prove it later. 
\begin{proof}[\textbf{Proof of $i)$ of Proposition \ref{prop:redu}}] Indeed, if $q(s_1) \in \partial V_A(s_1)$, we see from \eqref{equ:redu1} and the definition of $V_A(s)$ given in Proposition \ref{prop:1} that the first two components of $q(s_1)$ must be in $\partial \hat{V}_A(s_1)$, which is the conclusion of part $i)$ of Proposition \ref{prop:redu}, assuming Proposition \ref{prop:redu1} holds.
\end{proof}

\noindent We now give the proof of Proposition \ref{prop:redu1} in order to conclude the proof of part $i)$ of Proposition \ref{prop:redu}.
\begin{proof}[\textbf{Proof of Proposition \ref{prop:redu1}}] Note that the conclusion of this proposition is very similar to Proposition 3.7, pages 157 in \cite{MZdm97}. However, its proof is far from being an adaptation of the proof given in the case of the semilinear heat equation treated in \cite{MZdm97} because of the difference of the definition of $V_A(s)$ and the presence of the strong perturbation term. In fact, the argument given in \cite{MZdm97} does not work here to control $|q_2(s)|$ in this new situation, we use instead equation \eqref{equ:odeq2} to handle this term. \\

\noindent Let $\lambda_1 \geq \lambda_2$ be two positive numbers which will be fixed in term of $A$ later. It is enough to show that \eqref{equ:redu1} holds in two cases: $s - s_0 \leq \lambda_1$ and $s -s_0 \geq \lambda_2$. In both cases, we use Lemma \ref{lemm:estallterms} and equation and suppose $A \geq A_2 > 0$, $s_0 \geq \max\{\sigma_2(A,\lambda_1), \sigma_2(A,\lambda_2), \sigma_6(A), 1\}$. \\

\noindent \textbf{Case $s - s_0 \leq \lambda_1$}: Since we have for all $\tau \in [s_0,s]$, $q(\tau) \in V_A(\tau)$, we apply Lemma \ref{lemm:estallterms} with $A$ and $\lambda^* = \lambda_1$, and $\lambda = s - s_0$. From \eqref{for:qint} and Lemma \ref{lemm:estallterms}, we have 
$$|q_2(s)| \leq \frac{C}{s^2} + \frac{C\lambda_1}{s^{2 + \nu}}, \quad \quad \left\|\frac{q_-(y,s)}{1 + |y|^3} \right\|_{L^\infty}\leq \frac{C}{s^{3/2 + \varrho}}, \quad \|q_e(s)\|_{L^\infty} \leq \frac{Ce^{\lambda_1}}{\sqrt{s}} + \frac{C}{s^{\varrho}}.$$
If we fix $\lambda_1 = \frac{3}{2}\log A$ and $A$ large enough, then \eqref{equ:redu1} satisfies. \\

\noindent \textbf{Case $s- s_0 \geq \lambda_2$}: Since we have for all $\tau \in [\sigma, s]$, $q(\tau) \in V_A(\tau)$, we apply Lemma \ref{lemm:estallterms} with $A$, $\lambda = \lambda^* = \lambda_2$, $\sigma = s - \lambda_2$. From \eqref{for:qint} and Lemma \ref{lemm:estallterms}, we have 
\begin{align*}
\left\|\frac{q_-(y,s)}{1 + |y|^3} \right\|_{L^\infty} &\leq \frac{C}{s^{3/2 + \varrho}}\left(1 + A e^{-\frac{\lambda_2}{2}} + A^2e^{-\lambda_2^2}\right),\\
\|q_e(s)\|_{L^\infty} & \leq \frac{C}{s^\varrho}\left(1 + Ae^{\lambda_2} + A^2e^{-\frac{\lambda_2}{p}}\right).
\end{align*}
To obtain \eqref{equ:redu1}, except for $|q_2(s)|$ which will be treated later, it is enough to have $A \geq 4C$ and
\begin{align*}
& C\left(A e^{-\frac{\lambda_2}{2}} + A^2e^{-\lambda_2^2}\right) \leq \frac{A}{4},\\
& C\left(Ae^{\lambda_2} + A^2e^{-\frac{\lambda_2}{p}}\right) \leq \frac{A^2}{4}.
\end{align*}
If we fix $\lambda_2 = \log(A/8C)$ and take $A$ large enough, we see that these requests are satisfied. This follows the last two estimates in \eqref{equ:redu1}.\\

\noindent It remains to show that if $q(s) \in V_A(s)$ for all $s \in [s_0, s_1]$ then $|q_2(s)| < \frac{A^2}{s^{1 + \nu}}$ for all $s \in [s_0, s_1]$. We proceed by contradiction, assume that for all $s \in [s_0,s_*)$, $|q_2(s)| < \frac{A^2}{s^{1 + \nu}}$ and $|q_2(s_*)| = \frac{A^2}{s_*^{1 + \nu}}$. Considering the case $q_2(s_*) = -\frac{A^2}{s_*^{1 + \nu}}$, we have
\begin{equation}\label{equ:controlq2co}
q_2'(s_*) \leq \frac{d}{ds}\left( \frac{-A^2}{s_*^{1 + \nu}}\right) \leq \frac{(1 + \nu)A^2}{s_*^{2 + \nu}}.
\end{equation}
On the other hand, we have from \eqref{equ:odeq2},
$$q_2'(s_*) \geq -\frac 2s \;q_2(s_*) - \frac{C}{s_*^{2 + \nu}} = \frac{2A^2 -C}{s_*^{2 + \nu}},$$
which contradicts \eqref{equ:controlq2co} if we take $A$ large enough. \\
Using the same argument in the case where $q_2(s_*) = \frac{A^2}{s_*^{1 + \nu}}$, we also have a contradiction. This completes the proof of Proposition \ref{prop:redu1} and part $i)$ of Proposition \ref{prop:redu} too.
\end{proof}

\subsubsection*{Step 3: Deriving conclusion $ii)$ of Proposition \ref{prop:redu}.}
We prove part $ii)$ of Proposition \ref{prop:redu} here. In order to prove this, we follow the ideas of \cite{MZdm97} to show that for each $m \in \{0,1\}$ and each $\iota \in \{-1,1\}$, if $q_m(s_1) = \frac{\iota A}{s_1^{3/2 + \nu}}$, then $\frac{dq_m}{ds}(s_1)$ has the opposite sign of $\frac{d}{ds}\left(\frac{\iota A}{s^{3/2+\nu}}\right)(s_1)$ so that $(q_0,q_1)(s_1)$ actually leaves $\hat{V}_A$ at $s_1$ for $s_1 \geq s_0$ where $s_0$ will be large enough. Indeed, from equation \eqref{equ:eqm12}, we take $A = 2C + 1$. If $\iota = 1$, then $\frac{d q_m}{ds}(s_1) > 0$ and if $\iota = -1$, then $\frac{d q_m}{ds}(s_1) < 0$. This implies that $(q_0,q_1)(s_1 + \eta) \not\in \partial \hat{V}_A(s_1 + \eta)$ which yields conclusion $ii)$ of Proposition \ref{prop:redu}.

\subsection{Local in time solution of equation \eqref{equ:q}}
In the following, we find a local in time solution for equation \eqref{equ:q}.
\begin{prop}[\textbf{Local in time solution of equation \eqref{equ:q}}] \label{prop:localex} For all $A > 1$, there exists $\delta_5(A)$ such that for all $s_0 \geq \delta_5(A)$, the following holds: For all $(d_0,d_1) \in \mathcal{D}_{s_0}$, there exists $s_{max}(d_0,d_1) > s_0$ such that equation \eqref{equ:q} with initial data $q(s_0)$ given in \eqref{def:intq0} has a unique solution satisfying $q(s) \in V_{A+1}(s)$ for all $s \in [s_0,s_{max})$.
\end{prop}
\begin{proof} Using the definition  \eqref{def:varphi} of $q$ and the equivalent formulation \eqref{equ:simivariables}, we see that the Cauchy problem of \eqref{equ:q} is equivalent to the Cauchy problem of equation \eqref{equ:problem}. Note that the initial data for \eqref{equ:problem} is derived the initial data for \eqref{equ:q} at $s=s_0$ given in \eqref{def:intq0}, namely
$$u_{d_0,d_1}(x) = \frac{T^{-\frac{1}{p-1}}\phi(-\log T)}{\kappa}\left\{f(z)\left(1 + \frac{d_0 + d_1 z}{p-1 + \frac{(p-1)^2}{4p}z^2}\right) \right\},$$
where $f$ is defined in \eqref{def:f}, $T = e^{-s_0}$ and $z = \frac{x}{\sqrt{T|\log T|}}$.\\
This initial data is belong to $L^\infty(\R)$ which insures the local existence of $u$ in $L^\infty(\R)$ (see the introduction). From part $iii)$ of Lemma \ref{lemm:decomdata}, we have $q_{d_0,d_1}(s_0) \in V_A(s_0) \subseteq V_{A + 1}(s_0)$. Then there exists $s_{max}$ such that for all $s \in [s_0,s_{max})$, we have $q(s) \in V_{A+1}(s)$. This concludes the proof of Proposition \ref{prop:localex}.
\end{proof}
\subsection{Deriving conclusion $ii)$ of Theorem \ref{theo:1}}
In this subsection, we derive conclusion $ii)$ of Theorem \ref{theo:1} using the previous subsections. Although the derivation of the conclusion is the same as in \cite{MZdm97}, we would like to give details of its proof for the reader's convenience and for explaining the two-point strategy: reduction to a finite dimensional problem and the conclusion $ii)$ of Theorem \ref{theo:1} using index theory.
\begin{proof}[\textbf{Proof of $ii)$ of Theorem \ref{theo:1}}] We first solve the finite-dimensional problem and show the existence of $A > 1$, $s_0 > 0$ and $(d_0,d_1) \in \\mathcal{D}_{s_0}$ such that problem \eqref{equ:q} with initial data at $s = s_0$, $q_{d_0,d_1}(s_0)$ is given in \eqref{def:intq0} has a solution $q(s)$ defined for all $s \in [s_0,+\infty)$ such that 
\begin{equation}\label{equ:conii}
q(s) \in V_A(s), \quad \forall s \in [s_0,+\infty).
\end{equation}
For this purpose, let us take $A \geq A_1$ and $s_0 \geq \delta_3$, where $A_1$ and $\delta_3$ are given in Proposition \ref{prop:redu}, we will find the parameter $(d_0,d_1)$ in the set $\mathcal{D}_{s_0}$ defined in Lemma \ref{lemm:decomdata} such that \eqref{equ:conii} holds. We proceed by contradiction and assume from $iii)$ of Lemma \ref{lemm:decomdata} that for all $(d_0,d_1) \in \mathcal{D}_{s_0}$, there exists $s_*(d_0,d_1) \geq s_0$ such that $q_{d_0,d_1}(s) \in V_A(s)$ for all $s \in [s_0,s_*]$ and $q_{d_0,d_1}(s_*) \in \partial V_A(s_*)$. Applying Proposition \ref{prop:redu}, we see that $q_{d_0,d_1}(s_*)$ can leave $V_A(s_*)$ only by its first two components, hence,
$$(q_0,q_1)(s_*) \in \partial \hat{V}_A(s_*),$$
(see Proposition \ref{prop:1} for the definition of $\hat{V}_A$). Therefore, we can define the following function: 
\begin{align*}
\Phi \;: \mathcal{D}_{s_0} &\mapsto \partial ([-1,1]^2)\\
(d_0,d_1) &\to \frac{s_*^{1 + \nu}}{A}(q_0,q_1)(s_*).
\end{align*}
Since $q(y,s_0)$ is continuous in $(d_0,d_1)$ (see Lemma \ref{lemm:decomdata}), we have that $(q_0,q_1)(s)$ is continuous with respect to $(d_0,d_1,s)$. Then using the transversality property of $(q_0,q_1)$ on $\partial \hat{V}_A$ (part $ii)$ of Proposition \ref{prop:redu}), we claim that $s_*(d_0,d_1)$ is continuous. Therefore, $\Phi$ is continuous.\\
If we manage to prove that $\Phi$ is of degree one on the boundary, then we have a contradiction from the degree theory. Let us prove that. From Lemma \ref{lemm:decomdata}, we see that if $(d_0,d_1)$ is on the boundary of $\mathcal{D}_{s_0}$, then 
$$(q_0,q_1)(s_0) \in \partial \hat{V}_A(s_0), \quad \text{and} \quad q(s_0) \in V_A(s_0),$$
where the statement $q(s) \in V_A(s)$ holds with strict inequalities for $q_2, q_-$ and $q_e$. Using again $ii)$ of Proposition \ref{prop:redu}, we see that $q(s)$ can leave $V_A(s)$ at $s = s_0$, hence $s_* = s_0$. Using $iii)$ of Lemma \ref{lemm:decomdata}, we have that the restriction of $\Phi$ to the boundary is of degree $1$. This gives us a contradiction (by the index theory). Thus, there exists $(d_0,d_1) \in \mathcal{D}_{s_0}$ such that for all $s \geq s_0$, $q_{d_0,d_1}(s) \in V_A(s)$, which is the conclusion \eqref{equ:conii}.\\
Since $q_{d_0,d_1}(s)$ satisfies \eqref{equ:conii}, we use the parabolic estimate in Proposition \ref{prop:regW1in}, the transformations \eqref{def:varphi} and \eqref{equ:simivariables} and the fact that $\frac{\phi(s)}{\kappa} = 1 + \mathcal{O}(s^{-a})$ with $a > 0$ to derive estimate \eqref{est:th1}. This concludes the proof of $ii)$ of Theorem \ref{theo:1}. 
\end{proof}

\subsection{Deriving conclusion $iii)$ of Theorem \ref{theo:1}}
We give the proof of part $iii)$ of Theorem \ref{theo:1} in this subsection. We consider $u(t)$ solution of equation \eqref{equ:problem} which blows-up in finite-time $T > 0$ at only one blow-up point $x = 0$ and satisfies \eqref{est:th1}. Adapting the techniques used by Merle in \cite{Mercpam92} to equation \eqref{equ:problem} without the perturbation ($h \equiv 0$), we show the existence of a profile $u_* \in \mathcal{C}(\R \setminus \{0\}, \R)$ such that $u(x,t) \to u_*(x)$ as $t \to T$ uniformly on compact subsets of $\R \setminus \{0\}$, where $u_*$ is given in $iii)$ of Theorem \ref{theo:1}. Note that Zaag \cite{ZAAihn98}, Masmoudi and Zaag \cite{MZjfa08} successfully used these techniques to equation \eqref{equ:comGin}.  Since the proof is very similar to that written in \cite{ZAAihn98} and \cite{MZjfa08}, and no new idea is needed, we just give the key argument and kindly refer the reader to see Section 4 in \cite{ZAAihn98} for details.\\

\noindent For each $x_0 \in \R\setminus\{0\}$ small enough, we define  for all $(\xi,\tau) \in \R\times \left[-\frac{t(x_0)}{T-t(x_0)},1\right)$ the following function:
\begin{equation}\label{def:vres}
v(x_0,\xi,\tau) = (T -t(x_0))^\frac{1}{p-1}u(x_0+\xi\sqrt{T-t(x_0)}, t(x_0) + (T-t(x_0))\tau),
\end{equation}
where $t(x_0)$ is uniquely defined by 
\begin{equation}\label{def:tx0}
|x_0| = K_0\sqrt{(T - t(x_0))|\log (T - t(x_0))|},
\end{equation}
with $K_0 > 0$ to be fixed large enough later.\\
Note that $v$ blows up at time $\tau = 1$ at only one blow-up point $x_0 = 0$. From \eqref{equ:problem} and \eqref{def:vres}, we see that $v(x_0,\xi,\tau)$ satisfies the following equation: for all $\tau \in \left[-\frac{t(x_0)}{T-t(x_0)},1\right)$,
\begin{equation}\label{equ:vres}
\frac{\partial v}{\partial \tau} = \Delta_\xi v + |v|^{p-1}v + (T - t(x_0))^{\frac{p}{p-1}}h\left((T-t(x_0))^{-\frac{1}{p-1}}v\right).
\end{equation}
From estimate \eqref{est:th1}, the definition \eqref{def:vres} of $v$ and \eqref{def:tx0}, we have the following:
$$
\sup_{|\xi| < |\log (T - t(x_0))|^\frac{\varrho}{2}} |v(x_0,\xi,0) - f(K_0)| \leq \frac{C}{|\log(T-t(x_0))|^\frac{\varrho}{2}} \to 0 \quad \text{as} \;\; x_0 \to 0,
$$
where $f$ is given in \eqref{def:f}.\\
Using the continuity with respect to initial data for equation \eqref{equ:problem} (also for equation \eqref{equ:vres}) associated to a space-localization in the ball $B(0,|\xi| < |\log (T - t(x_0))|^\frac{\varrho}{2})$, it is showed in Section 4 of \cite{ZAAihn98} that 
$$
\sup_{|\xi| < |\log (T - t(x_0))|^\frac{\varrho}{2}, 0 \leq \tau < 1} |v(x_0,\xi,\tau) - \hat{f}_{K_0}(\tau)| \leq \epsilon(x_0) \to 0 \quad \text{as} \;\; x_0 \to 0,
$$
where $\hat{f}_{K_0}(\tau) = \kappa(1 - \tau + \frac{p-1}{4p}K_0^2)^{-\frac{1}{p-1}}$.\\
Then letting $\tau \to 1$ and using the definition \eqref{def:vres} of $v$, we have
\begin{align*}
u_*(x_0) = \lim_{t' \to T}u(x,t') &= (T - t(x_0))^{-\frac{1}{p-1}}\lim_{\tau \to 1}v(x_0,0,\tau)\\
&\sim (T - t(x_0))^{-\frac{1}{p-1}} \hat{f}_{K_0}(1), \quad \text{as}\;\; x_0 \to 0.
\end{align*}
From \eqref{def:tx0}, we have 
$$(T - t(x_0))^{-\frac{1}{p-1}} \sim \left( \frac{|x_0|^2}{2K_0^2|\log x_0|}\right)^{-\frac{1}{p-1}}, \quad \text{as}\;\; x_0 \to 0.$$
Hence,
$$u_*(x_0) \sim \left(\frac{8p |\log x_0|}{(p-1)^2|x_0|^2} \right)^\frac{1}{p-1}, \quad \text{as}\;\; x_0 \to 0,$$
which concludes the proof of part $iii)$ of Theorem \ref{theo:1}.

\appendix
\renewcommand*{\thesection}{\Alph{section}}
\counterwithin{theo}{section}
\section{Appendix A}
We claim the following:
\begin{lemm} \label{ap:lemmA2} Let $\varepsilon \in (0,p]$, there exist $C = C(a,p,\mu, M) > 0$ and $s_0 = s_0(a,\varepsilon) > 0$ such that for all $s \geq s_0$,\\
$i)$ if $h$ is given by \eqref{equ:h},
\begin{equation*}
j =0, 1,\;\; e^{-\frac{(p-j)s}{p-1}}\left|h^{(j)}\left(e^\frac{s}{p-1}w\right) \right| \leq  Cs^{-a}\left(|w|^{p-j} +1\right),
\end{equation*}
$ii)$ if $h$ is given by \eqref{equ:h1},
\begin{equation*}
\sum_{j = 0}^3 e^{-\frac{(p-j)s}{p-1}}|w|^j\left|h^{(j)}\left(e^\frac{s}{p-1}w\right) \right| \leq  Cs^{-a}(|w|^p + |w|^{p-\varepsilon}).
\end{equation*}
\end{lemm}
\begin{proof} We see that the proof directly follows from the following key estimate: 
\begin{equation}\label{est:keyH}
\frac{|w|^\varepsilon}{\log^a\left(2 + e^\frac{2s}{p-1}w^2 \right)} \leq \frac{C}{s^a}(|w|^\varepsilon + 1), \quad\forall s \geq s_0(a,\varepsilon).
\end{equation}
Indeed, considering the case $w^2e^\frac{s}{p-1} \geq 4$, we have 
$$\frac{|w|^\varepsilon}{\log^a\left(2 + e^\frac{2s}{p-1}w^2 \right)} \leq \frac{|w|^\varepsilon}{\log^a\left(4e^\frac{s}{p-1}\right)} \leq \frac{C|w|^\varepsilon}{s^a},$$
then the case $w^2e^\frac{s}{p-1} \leq 4$ which follows that 
$$\frac{|w|^\varepsilon}{\log^a\left(2 + e^\frac{2s}{p-1}w^2 \right)} \leq \frac{|w|^\varepsilon}{\log^a(2)} \leq Ce^{-\frac{\varepsilon s}{p-1}} \leq Cs^{-a}.$$
This concludes the proof of \eqref{est:keyH} and the proof of Lemma 
\ref{ap:lemmA2} also.
\end{proof}
The following lemma shows the existence of solutions of the associated ODE of equation \eqref{equ:divw1}:
\begin{lemm} \label{ap:lemmA3} Let $\phi$ be a positive solution of the following ordinary differential equation:
\begin{equation}\label{ap:phiODE}
\phi_s = -\frac{\phi}{p-1} + \phi^p + e^{-\frac{ps}{p-1}}h\left(e^\frac{s}{p-1} \phi\right).
\end{equation}
Then $\phi(s) \to \kappa$ as $s \to +\infty$ and $\phi(s)$ is given by
\begin{equation}\label{ap:solphi}
\phi(s) = \kappa(1 + \eta_a(s))^{-\frac{1}{p-1}}, \quad \text{where} \quad \eta_a(s) = \mathcal{O}\left(\frac{1}{s^a}\right).
\end{equation}
If $h(x) = \mu\frac{|x|^{p-1}x}{\log^a(2 + x^2)}$, then
$$\eta_a(s) \sim C_0\int_s^{+\infty}\frac{e^{s-\tau}}{\tau^a}d\tau = \frac{C_0}{s^a}\left(1 + \sum_{j\geq 1}\frac{b_j}{s^j}\right),$$
where $C_0 = \mu\left(\frac{p-1}{2}\right)^a$ and $b_j = (-1)^j\prod_{i = 0}^{j-1}(a+i)$.
\end{lemm}
\begin{proof} See Lemma A.3 in \cite{NG14the}.
\end{proof}
\section{Proof of Lemma \ref{lemm:BKespri}} \label{ap:2a}
In this appendix, we give the proof of Lemma \ref{lemm:BKespri}. The proof follows from the techniques of Bricmont and Kupiainen \cite{BKnon94} with some additional care, since we have a different profile function $\varphi$ defined in \eqref{def:varphi}, and since we give the explicit dependence of the bounds in terms of all the components of initial data. As mentioned early, the proof relies mainly on the understanding of the behavior of the kernel $\mathcal{K}(s,\sigma,y,x)$ (see \eqref{def:kernel}). This behavior follows from a perturbation method around $e^{(s-\sigma)\mathcal{L}}(y,s)$, where the kernel of $e^{t\mathcal{L}}$ is given by Mehler's formula:
\begin{equation}\label{for:kernalL}
e^{t\mathcal{L}}(y,x) = \frac{e^t}{\sqrt{4\pi (1 - e^{-t} )}} exp\left[-\frac{(ye^{-t/2} - x)^2}{4 (1 - e^{-t})}\right].
\end{equation}
By definition \eqref{def:kernel} of $\mathcal{K}$, we use a Feynman-Kac representation for $\mathcal{K}$:
\begin{equation}\label{for:kernelK}
\mathcal{K}(s,\sigma,y,x) = e^{(s-\sigma)\mathcal{L}}(y,x) \int d\mu_{yx}^{s-\sigma}(\omega)e^{\int_0^{s-\sigma}V(\omega(\tau), \sigma +\tau)d\tau},
\end{equation}
where $d\mu_{yx}^{s-\sigma}$ is the oscillator measure on the continuous paths $\omega: [0, s-\sigma] \to \mathbb{R}$ with $\omega(0) = x$, $\omega(s-\sigma) = y$, i.e. the Gaussian probability measure with covariance kernel 
\begin{align*}
\Gamma(\tau, \tau') &= \omega_0(\tau)\omega_0(\tau')\\
&+2 \left(e^{-\frac{1}{2}|\tau - \tau'|} -  e^{-\frac{1}{2}|\tau + \tau'|} + e^{-\frac{1}{2}|2(s - \sigma) + \tau - \tau'|} -  e^{-\frac{1}{2}|2(s - \sigma) - \tau - \tau'|} \right),
\end{align*}
which yields $\int d\mu_{yx}^{s-\sigma} \omega(\tau) = \omega_0(\tau)$, with 
$$\omega_0(\tau) = \left(\sinh((s-\sigma)/2) \right)^{-1}\left(y\sinh(\frac{\tau}{2}) + x\sinh(\frac{s-\sigma -\tau}{2}) \right).$$
In view of \eqref{for:kernelK}, we can consider the expression for $\mathcal{K}$ as a perturbation of $e^{(s-\sigma)\mathcal{L}}$. Since our profile $\varphi$ defined in \eqref{def:varphi} is different from the one defined in \cite{BKnon94}, we have here a potential $V$ defined in \eqref{def:V} which is different as well. Thus, we first estimate the potential $V$, then we restate some basic properties of the kernel $\mathcal{K}$.
\begin{lemm}[\textbf{Estimates on the potential $V$}] \label{ap:lemmEstV} For $s$ large enough, we have \\
$a)\;\;$ $V(y,s) \leq \frac{C}{s^{a'}}$ with $a' = \min\{a,1\}$.\\
$b)$ $\;\; \left|\frac{d^mV(y,s)}{d y^m} \right| \leq \frac{C}{s^{m/2}}$ for $m = 0,1,2$.\\
$c)\;\;$ $|V(y,s)| \leq \frac{C}{s}(1 + |y|^2)$, $\;\; V(y,s) = -\frac{h_2(y)}{4s} + \tilde{V}(y,s)$,\\
where $|\tilde{V}(y,s)| = \mathcal{O}\left(\frac{1 + |y|^4}{s^2}\right) + \mathcal{O}\left(\frac{1}{s^a} \right)$ in the case \eqref{equ:h} and $|\tilde{V}(y,s)| = \mathcal{O}\left(\frac{1 + |y|^4}{s^2}\right) + \mathcal{O}\left(\frac{1 + |y|^2}{s^{a+1}}\right)$ in the case \eqref{equ:h1}. In particular, in both cases $|\tilde{V}(y,s)| \leq \frac{C(1 + |y|^4)}{s^{1 + \bar{a}}}$, where $\bar{a} = \min\{a-1,1\}$ in the case \eqref{equ:h} and $\bar{a} = \min\{a,1\}$ in the case \eqref{equ:h1}.
\end{lemm}
\begin{proof} $a)$ From the definition \eqref{def:V} of $V$, we see that
$$V(y,s) \leq p (\varphi(0,s)^{p-1} - \kappa^{p-1}) + \imath \left|e^{-s}h'\left(e^\frac{s}{p-1}\varphi \right) \right|,$$
where $\imath$ is defined in \eqref{def:V}. From Lemma \ref{ap:lemmA3}, we have 
$$\varphi(0,s)^{p-1} - \kappa^{p-1} = \kappa^{p-1} \left[(1 + \eta_a(s))^{-1}\left(1 + \frac{1}{2ps}\right)^{p-1} - 1\right] \leq \frac{C}{s^{a'}}.$$
Since $|\varphi|$ is bounded, Lemma \ref{ap:lemmA2} yields $\imath \left|e^{-s}h'\left(e^\frac{s}{p-1}\varphi \right) \right| \leq \frac{\imath C}{s^a}$. This concludes part $a)$.\\
$b)$ We introduce $W(z,s) = V(y,s)$ with $z = \frac{y}{\sqrt{s}}$. In order to derive part $b)$, it is enough to show that $|\frac{d^m W}{d z^m}| \leq C$ for $m = 0,1,2$, which follows easily from Lemma \ref{ap:lemmA2} and the following key estimates 
$$\frac{\partial f(z)}{\partial z} = -\frac{z f(z)}{2p(1 + c_pz^2)},$$
where $f$ and $c_p$ are defined in \eqref{def:f}.\\
$c)$ Since $|V(y,s)| \leq C$ for all $y \in \mathbb{R}$ and $s \geq 1$, considering the cases $|y| \leq \sqrt{s}$, then $|y| \geq \sqrt{s}$, we directly see that the first estimate follows from the second. Hence, we only prove the second. To do so, we introduce $\tilde{W}(Z,s) = V(y,s)$ with $Z = \frac{|y|^2}{s}$. By the definition \eqref{def:varphi} and by a direct calculation, we find that 
\begin{align*} 
\frac{d^2\tilde{W}(Z,s)}{dZ^2} &= p(p-1)(p-2)\varphi^{p-3}(Z,s) \left(\frac{d \varphi(Z,s)}{dZ}\right)^2\\
& + \imath e^{-\frac{(p-3)s}{p-1}} h'''\left(e^\frac{s}{p-1}\varphi(Z,s)\right) \left(\frac{d \varphi(Z,s)}{dZ}\right)^2\\
& + \left[p(p-1)\varphi^{p-2}(Z,s) + \imath e^{-\frac{(p-2)s}{p-1}} h''\left(e^\frac{s}{p-1}\varphi(Z,s)\right) \right] \frac{d^2\varphi(Z,s)}{dZ^2}.
\end{align*}
Applying Lemma \ref{ap:lemmA2} with $\varepsilon = \frac{p-1}{2}$, we see that
\begin{align*}
\left|\frac{d^2\tilde{W}(Z,s)}{dZ^2}\right| &\leq C \left(\varphi^{p-3}(Z,s) +  \imath \varphi^{p-3 - \frac{p-1}{2}}(Z,s)\right)\left(\frac{d \varphi(Z,s)}{dZ}\right)^2\\
&+ C\left(\varphi^{p-2}(Z,s) + \imath \varphi^{p- 2-\frac{p-1}{2}}(Z,s)\right)\left|\frac{d^2\varphi(Z,s)}{dZ^2}\right|, \quad \forall s \geq s_0.
\end{align*}
From the definition \eqref{def:varphi} of $\varphi$, we note that $\frac{d \varphi}{dZ} = -\frac{c_p \phi}{\kappa}F^p(Z)$, where $c_p = \frac{p-1}{4p}$ and $F(Z) = \kappa(1 + c_pZ)^{-\frac{1}{p-1}}$, we derive
$$\varphi(Z,s)^{p-3 - \frac{p-1}{2}}\left(\frac{d \varphi(Z,s)}{dZ}\right)^2 \leq C\left(F + \frac{\kappa}{2ps} \right)^{p-3 - \frac{p-1}{2}}F^{2p} \leq 2C,$$
and 
$$\varphi(Z,s)^{p-2 - \frac{p-1}{2}}\left|\frac{d^2\varphi(Z,s)}{dZ^2}\right| \leq C\left(F + \frac{\kappa}{2ps} \right)^{p-2 - \frac{p-1}{2}}F^{2p}\leq 2C.$$
Hence, $\left|\frac{d^2\tilde{W}(Z,s)}{dZ^2}\right|$ is bounded for all $Z \in [0,+\infty)$ and for all $s \geq s_0$. Then by a Taylor expansion, we have 
$$\left|\tilde{W}(Z,s) - \tilde{W}(0,s) - Z\frac{\partial \tilde{W}(0,s)}{\partial Z}\right| \leq CZ^2, \quad \forall Z \in [0,+\infty), \;\; \forall s \geq s_0.$$
From the definition \eqref{def:varphi} of $\varphi$ and from Lemma \ref{ap:lemmA3}, we have 
\begin{align*}
W(0,s) &= \frac{p}{p-1}\left[\frac{1}{(1 + \eta_a)} \left(1+ \frac{1}{2ps} \right)^{p-1} -1\right] + \imath e^{-s}h'\left(e^\frac{s}{p-1}(\phi + \frac{\phi}{2ps}) \right)\\
&= \frac{1}{2s} - \frac{p}{p-1}\left(\frac{\eta_a(s)}{1 + \eta_a(s)}\right) + \imath e^{-s}h'\left(e^\frac{s}{p-1}\phi \right) + \mathcal{O}\left(\frac{1}{s^{a+1}} \right) + \mathcal{O}\left(\frac{1}{s^{2}} \right).
\end{align*}
Recalling from Lemma \ref{ap:lemmA3} that $\eta_a(s) = \mathcal{O}\left(\frac{1}{s^{a}}\right)$, this immediately yields $W(0,s) = \frac{1}{2s} + \mathcal{O}\left(\frac{1}{s^{a}}\right) + \mathcal{O}\left(\frac{1}{s^{2}}\right)$ in the case \eqref{equ:h}. In the case \eqref{equ:h1}, we obtain by a direct calculation,
\begin{align*}
&\left|- \frac{p}{p-1}\left(\frac{\eta_a(s)}{1 + \eta_a(s)}\right) + \imath e^{-s}h'\left(e^\frac{s}{p-1}\phi \right)\right|\\
&\qquad \qquad= \left|-\frac{p}{(p-1)(1 + \eta_a(s))} \left(\eta_a(s) - \frac{C_0}{s^a} \right)\right| + \mathcal{O}\left(\frac{1}{s^{a+1}} \right) = \mathcal{O}\left(\frac{1}{s^{a+1}} \right).
\end{align*}
In the last estimate, we used that fact that $\eta_a(s)= \frac{C_0}{s^a} +  \mathcal{O}\left(\frac{1}{s^{a+1}}\right)$ in the case \eqref{equ:h1} (see Lemma \ref{ap:lemmA3}). Hence, $W(0,s) = \frac{1}{2s} + \mathcal{O}\left(\frac{1}{s^{a+1}} \right) + \mathcal{O}\left(\frac{1}{s^{2}}\right)$ in the case \eqref{equ:h1}.\\
For $\frac{\partial W(0,s)}{\partial Z}$, we use Lemmas \ref{ap:lemmA2} and \ref{ap:lemmA3} to derive 
\begin{align*}
\frac{\partial W(0,s)}{\partial Z} &= -\frac{1}{4(1 + \eta_a(s))}\left(1 + \frac{1}{2ps} \right)^{p-2} - \imath e^{-\frac{(p-2)s}{p-1}}\frac{\phi}{4p}h''\left(e^\frac{s}{p-1}(\phi + \frac{\phi}{2ps})\right)\\
&= -\frac{1}{4} +  \mathcal{O}\left(\frac{1}{s^a} \right) + \mathcal{O}\left(\frac{1}{s}\right).
\end{align*}
Returning to $V$, we conclude part $c)$. This ends the proof of Lemma \ref{ap:lemmEstV}.
\end{proof}

In what follows, we denote $\int f(y)g(y) \rho(y) dy$ by $\lag f,g \rag$ and write $\chi(y,s) = \chi(s)$ ($\chi$ is defined in \eqref{def:chi}). Let us now recall some basic properties of the kernel $\mathcal{K}$ stated in \cite{BKnon94}:
\begin{lemm}[\textbf{Bricmont and Kupiainen \cite{BKnon94}}]\label{lemm:baskenelK} For all $s \geq \sigma \geq \max\{s_0,1\}$ with $s \leq 2\sigma$ and $s_0$ given in Lemma \ref{ap:lemmA2}, for all $(y,x) \in \R^2$, \\
$a)$ $|\K(s,\sigma,y,x)| \leq Ce^{(s - \sigma)\mL}(y,x)$.\\
$b)$ $\K(s,\sigma,y,x) = e^{(s - \sigma)\mL}(y,x)\left(1 + P_2(y,x) + P_4(y,x) \right)$, where 
$$|P_2(y,x)| \leq \frac{C(s - \sigma)}{s}(1 + |y| + |x|)^2,$$
$$ \text{and} \;\;|P_4(y,x)| \leq \frac{C(s - \sigma)(1 + s - \sigma)}{s^{2}} (1 + |y| + |x|)^4.$$
Moreover, $\left|\lag k_2, \left(\K(s,\sigma) - \left(\frac{\sigma}{s}\right)^2\right)h_2\rag \right| \leq \frac{C(s - \sigma) (1 + s - \sigma)}{s^{1 + \bar{a}}}$, with $\bar{a} = \min\{a-1,1\}$ in the case \eqref{equ:h} and $\bar{a} = \min\{a,1\}$ in the case \eqref{equ:h1}.\\ 
\noindent $c)$ $\|\K(s,\sigma)(1 - \chi)\|_{L^\infty} \leq Ce^{-\frac{(s - \sigma)}{p}}$.
\end{lemm}
\begin{proof} $a)$ From part $a)$ of Lemma \ref{ap:lemmEstV} and the definition \eqref{for:kernelK} of $\K$, we have
\begin{align*}
|\mathcal{K}(s,\sigma,y,x)| &\leq e^{(s-\sigma)\mathcal{L}}(y,x) \int d\mu_{yx}^{s-\sigma}(\omega)e^{\int_0^{s-\sigma}C(\sigma +\tau)^{-a'}d\tau}\\
&\leq Ce^{(s-\sigma)\mathcal{L}}(y,x) \int d\mu_{yx}^{s-\sigma} (\omega) \leq C e^{(s-\sigma)\mathcal{L}}(y,x),
\end{align*}
since $s \leq 2\sigma$ and $d\mu_{yx}^{s-\sigma}$ is a probability.\\
$b)$ The proof is exactly the same as the corresponding one written in \cite{BKnon94}. Although there is the difference of $\tilde{V}(y,s)$ given in part $c)$ of Lemma \ref{ap:lemmEstV}, this change does not affect the argument given in \cite{BKnon94}. For that reason, we refer the reader to Lemma 5, page 555 in \cite{BKnon94} for details of the proof.\\
$c)$ Our potential $V$ given in \eqref{def:V} has the same behavior as the potential in \cite{BKnon94} for $\frac{|y|^2}{s}$ and $s$ large (see \eqref{equ:asymV}). For that reason, we refer to Lemma, page 559 in \cite{BKnon94} for its proof.
\end{proof}

Before going to the proof of Lemma \ref{lemm:BKespri}, we would like to state some basic estimates which will be used frequently in the proof. 
\begin{lemm} For $K$ large enough, we have the following estimates:\\
$a)$ For any polynomial $P$, 
\begin{equation}\label{pro:kerL1}
\int P(y)\mathbf{1}_{\{|y| \geq K\sqrt{s}\}} \rho(y)dy \leq C(P)e^{-s}.
\end{equation}
$b)$ Let $p \geq 0$ and $|f(x)| \leq (1 + |x|)^p$, then 
\begin{equation}\label{pro:kerL2}
|(e^{t\mL}f)(y)| \leq C e^t(1 + e^{-t/2}|y|)^p,
\end{equation}
\end{lemm}
\begin{proof} $i)$ follows from a direct calculation. $ii)$ follows from the explicit expression \eqref{for:kernalL} by a simple change of variable. 
\end{proof}

\noindent Let us now give the proof of Lemma \ref{lemm:BKespri}.
\begin{proof}[\textbf{Proof of Lemma \ref{lemm:BKespri}}] We consider $\lambda > 0$, let $\sigma_0 \geq \lambda$, $\sigma \geq \sigma_0$ and $\psi(\sigma)$ satisfying \eqref{equ:boundpsisigma}. We want to estimate some components of $\theta(y,s) = \K(s,\sigma)\psi(\sigma)$ for each $s \in [\sigma, \sigma + \lambda]$. Since $\sigma \geq \sigma_0 \geq \lambda$, we have 
\begin{equation}\label{equ:ressigma}
\sigma \leq s \leq 2\sigma.
\end{equation}
Therefore, up to a multiplying constant, any power of any $\tau \in [\sigma, s]$ will be bounded systematically by the same power of $s$.

\noindent $a)$ \textbf{Estimate of $\theta_2$}: We first write 
\begin{align*}
\theta_2(s) &= \lag k_2, \chi(s)\K(s,\sigma)\psi(\sigma)\rag\\
&= \sigma^2s^{-2}\psi_2(\sigma) + \lag k_2, (\chi(s) - \chi(\sigma))\sigma^2s^{-2}\psi(\sigma)\rag \\
& \qquad + \lag k_2, \chi(s)(\K(s,\sigma)-\sigma^2s^{-2})\psi(\sigma)\rag := \sigma^2s^{-2}\psi_2(\sigma) + Ib + IIb.
\end{align*}
To bound $Ib$, we write $\psi(x,\sigma) = \sum_{l=0}^2\psi_l(\sigma)h_l(x) + \frac{\psi_-(x,\sigma)}{1+|x|^3} (1 + |x|^3) + \psi_e(x,\sigma)$ and  use \eqref{pro:kerL1} to derive
$$|Ib| \leq C (s-\sigma)e^{-s}\sigma^2s^{-2}\left(\sum_{l=0}^2|\psi_l(\sigma)| + \left\|\frac{\psi_-(x,\sigma)}{1+|x|^3}\right\|_{L^\infty}+\|\psi_e(\sigma)\|_{L^\infty}\right).$$
For $IIb$, we write 
\begin{align*}
IIb &= \sum_{l=0}^2 \lag k_2, \chi(s)(\K(s,\sigma) - \sigma^2s^{-2})h_l\rag\psi_l(\sigma)\\
&+ \lag k_2, \chi(s)(\K(s,\sigma) - \sigma^2s^{-2})\psi_-(\sigma)\rag \\
&+ \lag k_2, \chi(s)(\K(s,\sigma) - \sigma^2s^{-2})\psi_e(\sigma)\rag := IIb.1 + IIb.2 + IIb.3.
\end{align*}
Let us bound $IIb.1$. For $l = 2$, we already get from part $b)$ of Lemma \ref{lemm:baskenelK} and \eqref{pro:kerL1} that 
$$\left|\lag k_2, \chi(s)(\K(s,\sigma) - \sigma^2s^{-2})h_2\rag \psi_2(\sigma)\right| \leq \frac{C(s - \sigma)(1 + s - \sigma)}{s^{1 + \bar{a}}}|\psi_2(\sigma)|,$$
with $\bar{a} > 0$.\\
For $l = 0$ or $1$, we use $b)$ of Lemma \ref{lemm:baskenelK},  \eqref{pro:kerL2}, \eqref{pro:kerL1} and the fact that $\lag k_2,h_l\rag = 0$ and $e^{(s-\sigma)\mL}h_l = e^{(1 - l/2)(s - \sigma)}h_l$ to find that
\begin{align*}
\left|\lag k_2, \chi(s)(\K(s,\sigma) - \sigma^2s^{-2})h_l\rag \psi_l(\sigma)\right| &\leq  \left|\lag k_2, \chi(s)\left(\K(s,\sigma) - e^{(s-\sigma)\mL}\right)h_l\rag\right||\psi_l(\sigma)|\\
& + \left|\lag k_2, \chi(s)\left(e^{(s-\sigma)\mL} - \sigma^2s^{-2}\right)h_l\rag\right||\psi_l(\sigma)|\\
&\leq C(s-\sigma)\left(s^{-1} + e^{-s}\right)|\psi_l(\sigma)|\\
&\leq \frac{C(s-\sigma)}{s}|\psi_l(\sigma)|.
\end{align*}
This yields 
$$|IIb.1| \leq \frac{C(s - \sigma)}{s}\sum_{l=0}^2|\psi_l(\sigma)|.$$
If we write $\psi_-(x,\sigma) = \frac{\psi_-(x,\sigma)}{1+|x|^3}(1 + |x|^3)$ and use the same arguments as for $l = 0$, we obtain 
$$|IIb.2| \leq \frac{C(s - \sigma)}{s}\left\| \frac{\psi_-(x,\sigma)}{1+|x|^3}\right\|_{L^\infty}.$$
For $IIb.3$, we write 
\begin{align*}
IIb.3 &= \lag k_2, \chi(s)\left(\K(s,\sigma) - e^{(s-\sigma)\mL}\right)\psi_e(\sigma)\rag \\
& + \lag k_2, \chi(s)\left(e^{(s-\sigma)\mL}-1\right)\psi_e(\sigma)\rag+ \lag k_2, \chi(s) (1 - \sigma^2 s^{-2})\psi_e(\sigma)\rag.
\end{align*}
Using \eqref{pro:kerL1}, we bound the last term by $C(s- \sigma)e^{-\sigma}\|\psi_e(\sigma)\|_{L^\infty} \leq C(s- \sigma)e^{-s/2}\|\psi_e(\sigma)\|_{L^\infty}$ from \eqref{equ:ressigma}. For the second term, we write $e^{(s-\sigma)\mL} - 1 = \int_0^{s - \sigma} d\tau \mL e^{\tau \mL}$ and use the fact that
\begin{equation}\label{est:XsXsi}
\sup_{|y| \leq 2K\sqrt{s}, |x| \geq K \sqrt{\sigma}} e^{-\frac{|y|^2}{4} - \frac{(ye^{-\tau/2} - x)^2}{4(1 - e^{-\tau})}} \leq e^{-2s},
\end{equation}
for $K$ large enough, then it is also bounded by $C(s- \sigma)e^{-s}\|\psi_e(\sigma)\|_{L^\infty}$. For the first term, we use $b)$ of Lemma \ref{lemm:baskenelK}, \eqref{pro:kerL2} and again \eqref{est:XsXsi} to bound it by $C(s- \sigma)s^{-1}e^{-s}\|\psi_e(\sigma)\|_{L^\infty}$. This yields 
$$|IIb.3| \leq C(s - \sigma)e^{-s/2}\|\psi_e(\sigma)\|_{L^\infty}.$$
Collecting all these bounds yields the bound for $\theta_2(s)$ as stated in \eqref{equ:boundThe_2}.\\

\noindent $b)$ \textbf{Estimate of $\theta_-$:} By definition, 
\begin{align*}
\theta_-(y,s) &= P_-\left[\chi(s)\K(s,\sigma)\psi(\sigma)\right] = \sum_{l = 0}^2 \psi_l(\sigma)P_-\left[\chi(s)\K(s,\sigma)h_l\right]\\
&+ P_-\left[\chi(s)\K(s,\sigma)\psi_-(\sigma)\right] + P_-\left[\chi(s)\K(s,\sigma)\psi_e(\sigma)\right]:= Ic + IIc + IIIc.
\end{align*}
In order to bound $Ic$, we write $\K(s,\sigma) = \K(s,\sigma) - e^{(s - \sigma)\mL} + e^{(s - \sigma)\mL}$, then we use the fact that $e^{(s-\sigma)\mL}h_l = e^{(1 - l/2)(s - \sigma)}h_l$, part $b)$ of Lemma \ref{lemm:baskenelK} and \eqref{pro:kerL2} to derive for $l = 0,1,2$,
\begin{align*}
\left|\left(\K(s,\sigma) - e^{(s - \sigma)(1 - l/2)}\right)h_l \right|  &= \left|e^{(s-\sigma)\mL}\left(P_2 + P_4\right)h_l\right|\\
&\leq \frac{Ce^{s-\sigma}(s-\sigma)}{s}\left(1 + e^{-(s - \sigma)/2}|y|\right)^{2+l}\\
& + \frac{Ce^{s-\sigma}(s - \sigma)(1 + s - \sigma)}{s^{2}}\left(1 +e^{-(s - \sigma)/2}|y|\right)^{4 + l}.
\end{align*}
On the support of $\chi(s)$, namely when $|y| \leq 2K\sqrt{s}$, we can bound $s^{-k/2}|y|^k$ by $C$ for $k \in \mathbb{N}$. Then, from the easy-to-check fact that 
\begin{equation}\label{equ:propFP_}
\text{if}\;\; |f(y)| \leq m(1+ |y|^3), \;\;\text{then}\;\; P_-\left[f(y)\right] \leq Cm(1 + |y|^3),
\end{equation}
we obtain
\begin{align*}
l = 0, 1,\;\;\; &P_- \left[\psi_l(\sigma) \chi(s) \K(s,\sigma)h_l - \psi_l(\sigma)e^{(s - \sigma)(1 - l/2)}(\chi(s) h_l) \right]\\
& \qquad \qquad \qquad \qquad \leq \frac{Ce^{s - \sigma}(s - \sigma)(1 + s - \sigma)}{s}(1 + |y|^3)|\psi_l(\sigma)|,
\end{align*}
and 
\begin{align*}
&P_- \left[\psi_2(\sigma) \chi(s) \K(s,\sigma)h_2 - \psi_2(\sigma)e^{(s - \sigma)(1 - l/2)}(\chi(s) h_2) \right]\\
& \qquad \qquad \qquad \qquad \leq \frac{Ce^{s - \sigma}(s - \sigma)(1 + s - \sigma)}{\sqrt{s}}(1 + |y|^3)|\psi_2(\sigma)|.
\end{align*}
Since $P_-(h_l) = 0$ and $|(1 - \chi(y,s))h_l(y)| \leq Cs^{-3/2 + l/2}(1 + |y|^3)$, we have
$$l= 0, 1, 2, \;\;\left|\psi_l(\sigma)e^{(s-\sigma)(1 - l/2)}P_- \left[\chi(s) h_l(y)\right]\right| \leq \frac{Ce^{s-\sigma}}{s^{3/2-l/2}}|\psi_l(\sigma)|(1 + |y|^3).$$
Hence, 
$$|Ic| \leq \frac{Ce^{s - \sigma}\left((s - \sigma)^2 + 1\right)}{s}\left( |\psi_0(\sigma)| + |\psi_1(\sigma)| + \sqrt{s}|\psi_2(\sigma)|\right)(1 + |y|^3).$$
To bound $IIIc$, we use $a)$ of Lemma \ref{lemm:baskenelK} and the definition \eqref{for:kernalL} of $e^{(s - \sigma)\mL}$ to write
\begin{align*}
\left\|\frac{\chi(y,s)\K(s,\sigma) \psi_e(x,\sigma)}{1 + |y|^3} \right\|_{L^\infty} &\leq C e^{s - \sigma} \|\psi_e(\sigma)\|_{L^\infty}\\
&\sup_{|y| \leq 2K\sqrt{s},|x| \geq K\sqrt{\sigma}}e^{-\frac{1}{2}\frac{(ye^{-(s - \sigma)/2} - x)^2}{4(1 - e^{-(s - \sigma)})}}(1+|y|^3)^{-1}\\
&\leq \left\{\begin{array}{ll}
Cs^{-3/2}\|\psi_e(\sigma)\|_{L^\infty} & \text{if}\; s-\sigma \leq s_*\\
C e^{-s}\|\psi_e(\sigma)\|_{L^\infty}& \text{if}\; s-\sigma \geq s_*
\end{array} \right.
\end{align*}
for a suitable constant $s_*$.\\
Exploiting again \eqref{equ:propFP_}, we obtain the bound on this term which can be written as
$$|IIIc| \leq Cs^{-3/2}e^{-(s - \sigma)^2}\|\psi_e(\sigma)\|_{L^\infty}(1 + |y|^3) \;\; \text{for $\sigma$ large enough}.$$
We still have to consider $IIc$. In order to bound this term, we proceed as in \cite{BKnon94}. We write
\begin{equation}\label{equ:tmpKe}
\K(s,\sigma)\psi_-(\sigma) = \int dx e^{x^2/4}\K(s,\sigma)(\cdot,x)f(x) =\int dx N(\cdot,x) E(\cdot,x)f(x),
\end{equation}
where $f(x) = e^{-x^2/4}\psi_-(x,\sigma)$ and 
$$N(y,x) = \frac{e^{s - \sigma}e^{x^2/4}}{\sqrt{4\pi(1 - e^{-(s - \sigma)})}}e^{-\frac{(ye^{-(s - \sigma)/2} - x)^2}{4(1 - e^{-(s - \sigma)})}},$$
$$E(y,x) = \int d\mu^{s - \sigma}_{yx}(\omega)e^{\int_0^{s - \sigma}V(\omega(\tau),\sigma +\tau)d\tau}.$$
Let $f^0 = f$ and for $m \geq 1$, $f^{(-m-1)}(y) = \int_{-\infty}^y dx f^{(-m)}(x)$, then we have the following:
\begin{lemm}\label{lemm:estfm} $|f^{(-m)}(y)| \leq C\left\|\frac{\psi_-(x,\sigma)}{1+ |x|^3}\right\|_{L^\infty} (1 + |y|)^{(3-m)} e^{-y^2/4}$ for $m \leq 3$.
\end{lemm}
\begin{proof} See Lemma 6, page 557 in \cite{BKnon94}. 
\end{proof}
\noindent We now rewrite \eqref{equ:tmpKe} by integrating by parts as follows:
\begin{align}
\K(s,\sigma)\psi_-(\sigma) &= \sum_{l = 0}^2 (-1)^{l + 1}\int dx \partial_x^lN(y,x) \partial_x E(y,x)f^{(-l-1)}(x) \nonumber\\
& +\int dx \partial_x^3N(y,x) E(y,x)f^{-3}(x).\label{equ:tmpestKpar}
\end{align}
From the definition of $N(y,x)$, we have for $l = 0,1,2, 3$,
$$|\partial_x^lN(y,x)| \leq Ce^{-l(s-\sigma)/2}(1+ |y| + |x|)^le^{x^2/4}e^{(s - \sigma)\mL}(y,x).$$
Now using the integration by parts formula for Gaussian measures to write
\begin{align*}
&\partial_xE(y,x) = \frac{1}{2}\int_0^{s - \sigma}\int_0^{s - \sigma} d\tau d\tau' \partial_x \Gamma(\tau, \tau')\int d\mu_{yx}^{s - \sigma}(\omega)V'(\omega(\tau), \sigma+\tau)\\
&V'(\omega(\tau'), \sigma + \tau')e^{\int_0^{s-\sigma}d\tau''V(\omega(\tau''),\sigma + \tau'')}\\
& +\frac{1}{2}\int_0^{s - \sigma}d\tau \partial_x\Gamma(\tau,\tau)\int d\mu_{yx}^{s - \sigma}(\omega)V''(\omega(\tau), \sigma + \tau)e^{\int_0^{s - \sigma}d\tau''V(\omega(\tau''),\sigma + \tau'')}.
\end{align*}
Recalling from Lemma \ref{ap:lemmEstV} that $V(y,s) \leq \frac{C}{s^{a'}}$ with $a' > 0$ and $\left|\frac{d^mV(y,s)}{dy^m} \right| \leq \frac{C}{s^{m/2}}$ for $m = 0, 1, 2$. Since $s \leq 2\sigma$, this yields $\int_0^{s - \sigma}V(\omega(\tau), \sigma + \tau)d\tau \leq C$. Because $d\mu^{s-\sigma}_{yx}$ is a probability, we then obtain
$$|E(y,x)| \leq C \quad \text{and} \quad |\partial_xE(y,x)| \leq \frac{C}{s}(s - \sigma)(1 + s - \sigma)(|y| + |x|).$$
Substituting all these bounds into \eqref{equ:tmpestKpar}, then using \eqref{pro:kerL2}, Lemma \ref{lemm:estfm}, the fact that $s^{-1}(s - \sigma)(1 + s - \sigma) \leq e^{-3/2(s - \sigma)}$ for $s$ large and then \eqref{equ:propFP_}, we derive 
$$|IIc| \leq Ce^{-(s - \sigma)/2}\left\|\frac{\psi_-(x,\sigma)}{1+ |x|^3}\right\|_{L^\infty} (1 + |y|^3).$$
Collecting all the bounds for $Ic, IIc$ and $IIIc$, we obtain the bound \eqref{equ:boundThe_ne}.

\noindent $c)$ \textbf{Estimate for $\theta_e$}: By definition, we write
$$\theta_e(y,s) = (1 - \chi(y,s))\K(s,\sigma)\psi(\sigma) = (1 - \chi(y,s))\K(s,\sigma) \left(\psi_b(\sigma)+ \psi_e(\sigma)\right).$$
Using $c)$ of Lemma \ref{lemm:baskenelK}, we have
$$\|(1 - \chi(y,s))\K(s,\sigma)\psi_e(\sigma)\|_{L^\infty} \leq Ce^{-(s - \sigma)/p}\|\psi_e(\sigma)\|_{L^\infty}.$$
It remains to bound $(1 - \chi(y,s))\K(s,\sigma) \psi_b(\sigma)$. To this end, we write 
$$\psi_b(x,\sigma) = \sum_{l=0}^2\psi_l(\sigma)h_l(x) + \frac{\psi_-(x,\sigma)}{1 + |x|^3}(1 + |x|^3),$$
then we use $\chi(x,\sigma)|x|^k \leq C\sigma^{k/2} \leq Cs^{k/2}$ for $k \in \mathbb{N}$, and $a)$ of Lemma \ref{lemm:baskenelK} to derive
\begin{align*}
\|(1 - \chi(y,s))\K(s,\sigma) \psi_b(x,\sigma)\|_{L^\infty} &\leq Ce^{s - \sigma}\sum_{l=0}^2s^{l/2}|\psi_l(\sigma)|\\
& + Ce^{s - \sigma}s^{3/2}\left\|\frac{\psi_-(x,\sigma)}{1 + |x|^3}\right\|_{L^\infty}.
\end{align*}
This yields the bound \eqref{equ:boundThe_e} and concludes the proof of Lemma \ref{lemm:BKespri}.
\end{proof}
\section{Proof of Lemma \ref{lemm:estonBRN}} \label{ap:lemmEstAll}
We give the proof of Lemma \ref{lemm:estonBRN} here.
\begin{proof}[\textbf{Proof of Lemma \ref{lemm:estonBRN}}] $i)$ From the definition \eqref{def:B} of $B$, we use a Taylor expansion and the boundedness of $|\varphi|$ and $|q|$ to find that
\begin{equation}\label{equ:TestonB}
|\chi(\tau) B(q(\tau))| \leq C|q(\tau)|^2 \quad \text{and} \quad |B(q(\tau))| \leq C|q(\tau)|^{p'},
\end{equation}
where $p' = \min\{2,p\}$. \\
(Since we have the same definition of $B$ as in \cite{MZdm97}, we do not give the proof of \eqref{equ:TestonB} and kindly refer the reader to Lemma 3.15, page 168 of \cite{MZdm97} for its proof.)\\
Using \eqref{equ:TestonB} and \eqref{iq:VAbe}, we have
\begin{align}
|\chi(\tau) B(q(\tau))| \leq \frac{CA^4}{\tau^{3 + 2\varrho}}(1 + |y|^6) + \frac{CA^4}{\tau^{2 + 2\nu}}(1 + |y|^4).\label{equ:estTmponB}
\end{align}
From \eqref{equ:estTmponB}, we then derive for $m = 0,1,2$,
\begin{equation}\label{equ:apestBm}
|B_m(\tau)| = \left|\int \chi(\tau) B(q(\tau)) k_m \rho dy \right| \leq \frac{CA^4}{\tau^{2 + 2\nu}}.
\end{equation}
Since $B_-(y,\tau) = \chi(\tau) B(q(\tau)) - \sum_{m = 0}^2B_m(\tau)h_m(y)$, we have from \eqref{equ:estTmponB} and \eqref{equ:apestBm}, 
\begin{align*}
\left|\frac{B_-(y,\tau)}{1+|y|^3}\right|& \leq \left|\frac{\chi(\tau) B(q(\tau))}{1+|y|^3}\right| + \left|\frac{\sum_{m = 0}^2B_m(\tau)h_m(y)}{1 +|y|^3} \right|\\
& \leq \chi(\tau)\left[\frac{CA^4}{\tau^{3 + 2\varrho}}(1 +|y|^3) + \frac{CA^4}{\tau^{2 + 2\nu}}(1 +|y|)\right] + \frac{CA^4}{\tau^{2 + 2\nu}} \left( \frac{|\sum_{m = 0}^2 h_m(y)|}{1 + |y|^3} \right)
\end{align*}
If we use $|y|^l\chi(y,\tau)\leq C\tau^{l/2}$ for $l \in \mathbb{N}$, and $|\sum_{m = 0}^2 h_m(y)| \leq C(1 + |y|^2)$, then we obtain
$$\left\|\frac{B_-(y,\tau)}{1+|y|^3}\right\|_{L^\infty} \leq \frac{CA^4}{\tau^{3/2 + 2\varrho}}.$$
Using the second estimates in \eqref{equ:TestonB} and  \eqref{iq:VAbe}, we obviously obtain $\|B(\tau)\|_{L^\infty} \leq \frac{CA^{2p'}}{\tau^{\varrho p'}}$ which yields $\|B_e(\tau)\|_{L^\infty} \leq \frac{CA^{2p'}}{\tau^{\varrho p'}}$. This ends the proof of part $i)$.\\

\noindent $ii)$ From the definition \eqref{def:R} of $R$, we write $\varphi(y,\tau) = \frac{\phi(\tau)}{\kappa} \vartheta(y,\tau)$ and $R(y,\tau) = \frac{\phi(\tau)}{\kappa}Q + G$, where $\vartheta(y,\tau) = f(\frac{y}{\sqrt{\tau}}) + \frac{\kappa}{2p\tau}$ and
\begin{align}
Q(y,\tau) &= -\vartheta_\tau + \Delta \vartheta - \frac{y}{2}\nabla \vartheta - \frac{\vartheta}{p-1} + \vartheta^p,\label{def:QofR}\\
G(y,\tau) &= -\frac{\phi'}{\kappa}\vartheta - \frac{\phi}{\kappa}\vartheta^p + \phi^p\left(\frac{\vartheta}{\kappa}\right)^p + e^{\frac{-p\tau}{p-1}}h\left(e^\frac{\tau}{p-1}\frac{\phi}{\kappa} \vartheta\right).\label{def:GofR}
\end{align}
The conclusion of part $ii)$ is a direct consequence of the following:
\begin{lemm}\label{lemm:estonQR} There exists $\sigma_7 > 0$ such that for all $\tau \geq \sigma_7$, we have\\
$i)$ (\textbf{Estimates on $Q$})
\begin{align}
m = 0,1,\;&|Q_m(\tau)| \leq \frac{C}{\tau^2}, \quad |Q_2(\tau)| \leq \frac{C}{\tau^3},\nonumber\\
&\left\|\frac{Q_-(y,\tau)}{1 + |y|^3}\right\|_{L^\infty} \leq \frac{C}{\tau^2}, \quad \|Q_e(\tau)\|_{L^\infty} \leq \frac{C}{\sqrt{\tau}}.\label{est:QofR}
\end{align}
$ii)$ (\textbf{Estimates on $G$})
\begin{equation}\label{est:GofR}
m = 0,1,2,\;|G_m(\tau)| \leq \frac{C}{\tau^{1 + a'}},\; \left\|\frac{G_-(y,\tau)}{1 + |y|^3} \right\|_{L^\infty} \leq \frac{C}{\tau^{1 + a'}},\; \|G_e(\tau)\|_{L^\infty} \leq \frac{C}{\tau^a},
\end{equation}
where $a' = a > 1$ in the case \eqref{equ:h} and $a' = a + 1 > 1$ in the case \eqref{equ:h1}.
\end{lemm}
\begin{proof} $i)$ See page 563 in \cite{BKnon94}. For part $ii)$, one can see that it is a direct consequence of the following:
\begin{equation}\label{equ:estonG}
|G(y,\tau)| \leq \frac{C}{\tau^{a}}\quad \text{and}\quad |\chi(\tau) G(y,\tau)| \leq \frac{C}{\tau^{1 + a'}}(1 + |y|^2).
\end{equation}
By the definition of $G_m, G_-$ and $G_e$, part $ii)$ simply follows from \eqref{equ:estonG}. By the linearity, this also concludes the proof of part $ii)$ of Lemma \ref{lemm:estonBRN}.\\

\noindent Let us now give the proof of \eqref{equ:estonG}. For the first estimate, we use the definition \eqref{def:GofR} of $G$, Lemmas \ref{ap:lemmA2} and  \ref{ap:lemmA3},
\begin{align*}
|G(y,\tau)| \leq \left|\frac{\phi' \vartheta}{\kappa} \right| + \left|\frac{\phi \vartheta}{\kappa}\right| \left|1 - \frac{\phi^{p-1}}{\kappa^{p-1}} \right| + \left|e^{-\frac{ps}{p-1}}h\left(e^\frac{s}{p-1}\frac{\phi \vartheta}{\kappa} \right) \right| \leq \frac{C}{s^a}.
\end{align*}
For the second estimate in \eqref{equ:estonG}, we use the fact that $\phi$ satisfies \eqref{equ:phiODE} and write
\begin{align*}
G(y,\tau) &= \frac{\vartheta \phi}{\kappa^p}(\kappa^{p-1} - \phi^{p-1})(\kappa^{p-1} - \vartheta^{p-1})\\
& + e^{-\frac{p\tau}{p-1}}\left[h\left(e^\frac{\tau}{p-1}\frac{\phi \vartheta}{\kappa} \right) - h\left(e^\frac{\tau}{p-1} \phi \right)  \right]\\
& + \left( 1 - \frac{\vartheta}{\kappa}\right)e^{-\frac{p\tau}{p-1}}h\left(e^\frac{\tau}{p-1} \phi \right) := \bar{G} + \tilde{G} + \hat{G}.
\end{align*}
Noting that $\vartheta(y,\tau) = \kappa\left(1 - \frac{h_2(y)}{4p\tau} + \mathcal{O}\left( \frac{|y|^4}{\tau^2}\right) \right)$ uniformly for $y \in \mathbb{R}$ and $\tau \geq 1$, and recalling from Lemma \ref{ap:lemmA3} that $\phi(\tau) = \kappa(1 + \eta_a(\tau))^{-\frac{1}{p-1}}$ where $\eta_a(\tau) = \mathcal{O}(\tau^{-a})$, then using a Taylor expansion, we derive
\begin{align*}
 \bar{G}(y,\tau) &= \frac{\phi\eta_a(\tau)}{1 + \eta_a(\tau)} \left(\frac{h_2(y)}{4p\tau} + \mathcal{O}\left( \frac{|y|^4}{\tau^2}\right)\right) ,\\
\tilde{G}(y,\tau) &=  -\phi e^{-\tau}h'\left(e^\frac{\tau}{p-1} \phi\right)\left(\frac{h_2(y)}{4p\tau} + \mathcal{O}\left(\frac{|y|^4}{\tau^2}\right)\right),\\
\hat{G}(y,\tau) &= e^{-\frac{p\tau}{p-1}}h\left(e^\frac{\tau}{p-1}\phi\right) \left(\frac{h_2(y)}{4p\tau} + \mathcal{O}\left(\frac{|y|^4}{\tau^2}\right)\right).
\end{align*}
This yields the second estimate in \eqref{equ:estonG} in the case \eqref{equ:h}. If $h$ is given by \eqref{equ:h1}, we have furthermore
\begin{align*}
&\left|\frac{\phi\eta_a(\tau)}{1 + \eta_a(\tau)} - e^{-s}h'(e^\frac{s}{p-1}\phi)\phi  + e^{-\frac{ps}{p-1}}h(e^\frac{s}{p-1}\phi)\right| \\
&\leq \left|\frac{\phi}{1 + \eta_a(\tau)}\right|\left|\eta_a(\tau)- \frac{\mu}{\log^a\left(2 + e^\frac{2\tau}{p-1}\phi^2(\tau)\right)}\right| + \frac{C}{\tau^{a+1}} \leq \frac{2C}{\tau^{1 + a}},
\end{align*}
which yields the second estimate in \eqref{equ:estonG} in the case \eqref{equ:h1}. This concludes the proof of \eqref{equ:estonG} and the proof of part $ii)$ of Lemma \ref{lemm:estonBRN} also.
\end{proof}

\noindent $iii)$ From the definition \eqref{def:N} of $N$, we use a Taylor expansion for $N$ to find that in the case \eqref{equ:h},
$$N(q(\tau),\tau) = e^{-\tau}h'\left(e^\frac{\tau}{p-1}(\phi(\tau) + \theta_1 q(\tau))\right)q(\tau)\quad \text{with} \quad \theta_1 \in [0,1],$$
and in the case \eqref{equ:h1},
$$N(q(\tau),\tau) = e^{-\frac{(p-2)\tau}{p-1}}h''\left(e^\frac{\tau}{p-1}(\phi(\tau) + \theta_2 q(\tau))\right)q^2(\tau)\quad \text{with} \quad\theta_2 \in [0,1].$$
Since $\varphi(\tau) \to \kappa$ and $\|q(\tau)\|_{L^\infty(\mathbb{R})} \to 0$ as $\tau \to +\infty$, this implies that there exists $\tau_0$ large enough such that $\frac{\kappa}{2} \leq |\phi(\tau) + \theta_i q(\tau)| \leq \frac{3\kappa}{2}$ for all $\tau \geq \tau_0$ and $y \in \mathbb{R}$. Then by Lemma \ref{ap:lemmA2}, we have $|N(q(\tau),\tau)| \leq \frac{C |q|^\beta}{\tau^a}$ where $\beta = 1$ in the case \eqref{equ:h} and $\beta = 2$ in the case \eqref{equ:h1}, which implies part $iii)$ of Lemma \ref{lemm:estonBRN}. This concludes the proof of Lemma \ref{lemm:estonBRN}.
\end{proof}

\def\cprime{$'$}

\vspace*{0.4cm}
\noindent $\begin{array}{ll} \textbf{Address:} & \text{Universit\'e Paris 13, Institut Galil\'ee, LAGA,}\\
& \text{99 Avenue Jean-Baptiste Cl\'ement,}\\
& \text{93430 Villetaneuse, France.}\\
\end{array}$

\vspace*{0.2cm}
\noindent $\begin{array}{ll}\textbf{E-mail:}& \text{vtnguyen@math.univ-paris13.fr}\\
& \text{Hatem.Zaag@univ-paris13.fr}
\end{array}$


\end{document}